\documentclass[12pt]{article}
\usepackage[centertags]{amsmath}
\usepackage{amsfonts}
\usepackage{amssymb}
\usepackage{latexsym}
\usepackage{amsthm}
\usepackage{newlfont}
\usepackage{graphicx}
\usepackage{listings}
\usepackage{booktabs}
\usepackage{abstract}
\usepackage{xcolor}
\usepackage{enumerate}

\usepackage{cancel}

\date{}

\newlength{\defbaselineskip}
\newcommand{\setlinespacing}[1]%
           {\setlength{\baselineskip}{#1 \defbaselineskip}}

\newcommand{\N}{{\mathbb{N}}}

\newcommand{\actaqed}{\hfill $\actabox$}
{\medskip\noindent \textit{Proof of #1. }}%
{\actaqed \medskip}

\def\cD{{\mathcal D}}

\def\cC{{\mathcal C}}

\def\cH{{\mathcal H}}

\def \cV{\mathcal V}
\def \cU{\mathcal U}

\def \CW{\mathcal W}

\def \cM{\mathcal M}
\def\R{{\mathbb R}}
\def\RR{{\mathbb R}}
\def\EE{{\mathbb E}}

\def\PP{{\mathbb P}}

\def \<{\langle}
\def\>{\rangle}

\def \ep{\epsilon}
\def \e{\epsilon}
\def \va{\varepsilon}
\def\al{\alpha}
\def \de{\delta}

\def\vi{\varphi}
\def\la{\lambda}
\def\La{\Lambda}

\def \sa{\sigma}

\def\Og{\Omega}
\def\Ld{\Lambda}

\def\bx{\mathbf x}
\def\by{\mathbf y}
\def\bz{\mathbf z}

\def\bv{\mathbf v}
\def\bu{\mathbf u}

\def\bt{\beta}

\def\Br{\Bigr}
\def\Bl{\Bigl}
\def\f{\frac}

\def\NN{\mathbb{N}}

\newtheorem{Theorem}{Theorem}[section]
\newtheorem{Lemma}{Lemma}[section]

\newtheorem{Definition}{Definition}[section]
\newtheorem{Proposition}{Proposition}[section]
\newtheorem{Remark}{Remark}[section]

\newtheorem{Corollary}{Corollary}[section]

\numberwithin{equation}{section}

\newcommand{\be}{\begin{equation}}
\newcommand{\ee}{\end{equation}}
\newcommand{\wt}{\widetilde}

\def\sub{\substack}
\def\sa{\sigma}

\def\diam{\text{diam}}
\def\ga{\gamma}
\def\Ga{\Gamma}
\def\al{\alpha}
\def\cA{\mathcal{A}}

\begin{document}

		\title{Some improved bounds in sampling  discretization of integral norms}
		
		\author{ F. Dai,  E. Kosov,    V. Temlyakov 	\footnote{
							The first named author's research was partially supported by NSERC of Canada Discovery Grant
							RGPIN-2020-03909.
The second named author is a Young
Russian Mathematics award winner and would like to thank its sponsors and jury.
The second named author's research was prepared within the framework of the
HSE University Basic Research Program.
							The third named author's research was supported by the Russian Federation Government Grant No. 14.W03.31.0031.}}

		\newcommand{\Addresses}{{
				\bigskip
				\footnotesize
					\noindent	F.~Dai \\
				 \textsc{ Department of Mathematical and Statistical Sciences\\
					University of Alberta, Edmonton, Alberta T6G 2G1, Canada\\
					E-mail:} \texttt{fdai@ualberta.ca }
				
									\medskip
										\noindent	E. Kosov \\
									\textsc{Steklov Mathematical Institute of Russian Academy of Sciences, Moscow, Russia; National Research University Higher School of Economics, Russian Federation\\
										E-mail:} \texttt{ked$_{-}$2006@mail.ru}
									\medskip
									
				\noindent	V.N. Temlyakov\\
				 \textsc{University of South Carolina, 1523 Greene St., Columbia SC, 29208, USA;\\  Steklov Institute of Mathematics;    Lomonosov Moscow State University,\\
					 and Moscow Center for Fundamental and Applied Mathematics.
					\\
					E-mail:} \texttt{temlyak@math.sc.edu}
					}}
		\maketitle

	\begin{abstract}
		{The paper addresses a problem of sampling discretization of integral norms of elements of finite-dimensional subspaces satisfying some conditions. We prove sampling discretization results under a standard assumption formulated in terms of the Nikol'skii-type inequality.
{In particular, we obtain} some upper bounds on the number of sample points sufficient for good discretization of the integral  $L_p$ norms, $1\le p<2$, of functions from finite-dimensional subspaces of continuous functions. Our new results improve upon the known results in this direction. We use a new technique based on deep
			results of Talagrand from functional analysis.
}
	\end{abstract}
	
	{\it Keywords and phrases}: Sampling discretization,  Nikol'skii inequality, random vectors,  frame.
	
	{\it MSC classification 2000:} Primary 65J05; Secondary 42A05, 65D30, 41A63.

	\section{Introduction}
	\label{I}
	
	Let $\Omega$ be a  compact subset of $\R^d$ with the probability measure $\mu$. By $L_p$ norm, $1\le p< \infty$,  we understand
	$$
	\|f\|_p:=\|f\|_{L_p(\Omega)}:=\|f\|_{L_p(\Omega,\mu)} := \left(\int_\Omega |f|^pd\mu\right)^{1/p}.
	$$
	By $L_\infty$ norm we understand the uniform norm of continuous functions
	$$
	\|f\|_\infty := \max_{\bx\in\Omega} |f(\bx)|
	$$
	and with some abuse of notation we occasionally write $L_\infty(\Omega)$ for the space $\cC(\Omega)$ of continuous functions on $\Omega$.
	
	By discretization of the $L_p$ norm we understand a replacement of the measure $\mu$ by
	a discrete measure $\mu_m$ with support on a set $\xi =\{\xi^j\}_{j=1}^m \subset \Omega$   in such a way that the error $|\|f\|^p_{L_p(\Omega,\mu)} - \|f\|^p_{L_p(\Omega,\mu_m)}|$ is small for functions from a given class. In this paper we focus on discretization of the $L_p$ norms of elements of finite-dimensional subspaces. Namely, we work on the following problem.
	
	{\bf The Marcinkiewicz discretization problem.} Let $\Omega$ be a subset of $\R^d$ with the probability measure $\mu$. We say that a linear subspace $X_n$ (the index $n$ here, usually, stands for the dimension of $X_n$) of $L_p(\Omega,\mu)$, $1\le p < \infty$, admits the Marcinkiewicz-type discretization theorem with parameters $m\in \N$ and $p$ and positive constants $C_1\le C_2$ if there exists a set
	$$
	\Big\{\xi^j \in \Omega: j=1,\dots,m\Big\}
	$$
	such that for any $f\in X_n$ we have
	\be\label{I1}
	C_1\|f\|_p^p \le \frac{1}{m} \sum_{j=1}^m |f(\xi^j)|^p \le C_2\|f\|_p^p.
	\ee
%

	{\bf The Marcinkiewicz discretization problem with $\e\in (0, 1)$.} We write $X_n\in \cM(m,p,\e)$ if (\ref{I1}) holds with $C_1=1-\e$ and $C_2=1+\e$.

	There are known results on the Marcinkiewicz discretization problem proved for subspaces $X_n$
	satisfying some conditions. There are two types of conditions used in the literature: (I) Conditions on the entropy numbers and (II) Conditions in terms of the Nikol'skii-type inequalities. The reader can find a detailed discussion of known results in the very recent survey \cite{KKLT}. In this paper we only prove some discretization results under conditions (II).
	We now describe these conditions in detail.
	
	{\bf Nikol'skii inequalities.} Let $q\in [1,\infty)$ and $X_n\subset L_\infty(\Omega)$. The inequality
	\begin{equation}\label{I4}
	\|f\|_\infty \leq C\|f\|_q,\   \ \forall f\in X_n
	\end{equation}
	is called the Nikol'skii inequality for the pair $(q,\infty)$ with the constant $C$. In this paper
	it is convenient for us to write the constant $C$ in the form $C=(Kn)^{1/q}$. We obtain here discretization results under the Nikol'skii inequality for the pair $(2,\infty)$. In Section \ref{p1} we
	prove the following Theorem~\ref{MT1}, which is one of the main results of the paper.
	
	
	\begin{Theorem}\label{MT1}
		There exists a positive absolute constant  $C$  such that for any subspace $X_n$ of $\cC(\Og)$  of dimension at most $n$
		satisfying the Nikol'skii inequality
		\begin{equation}\label{M1}
		\|f\|_{\infty} \leq \sqrt{Kn} \|f\|_{L_2(\Og,\mu)},\     \  \   \ \forall f\in X_n
		\end{equation}
		 for some  probability measure $\mu$,
		and for any  $\e \in (0,1)$,
		there is a finite set of points  $\{\xi^1,\dots,\xi^m\}\subset \Og$  with
		$$
		m\leq  C\e^{-2} Kn\log n,
		$$
		which provides the following discretization inequalities for any $f\in X_n$
		\be\label{M2}
		{(1-\e)\|f\|_{L_1(\Og,\mu)} \leq \f 1m \sum_{j=1}^m |f(\xi^j)| \leq  (1+\e)\|f\|_{L_1(\Og,\mu)},}
		\ee
		\be\label{M2.1}
	{(1-\e)\|f\|_{L_2(\Og,\mu)}^2 \leq \f 1m \sum_{j=1}^m |f(\xi^j)|^2 \leq  (1+\e)\|f\|_{L_2(\Og,\mu)}^2.}
		\ee
	\end{Theorem}
	
	Theorem \ref{MT1} guarantees good discretization with equal weights under the Nikol'skii inequality for the pair $(2,\infty)$ with the bound on the number of points $m\leq C_{1}\e^{-2} Kn\log n$.
	This covers the case of simultaneous discretization of the $L_1$ and $L_2$ norms. In the case of
	simultaneous discretization of the $L_p$, $1<p<2$, and $L_2$ norms the following Theorem \ref{MT2}, {which is proved in Section \ref{p}}, provides a little worse guarantees on the number of points for good discretization. Theorem~\ref{MT2} is the second main result of the paper.
	
	\begin{Theorem}\label{MT2} Let $1<p<2$.  There exists a positive constant  $C(p)$
        such that for any subspace $X_n$ of $\cC(\Og)$ of dimension at most $n$
		satisfying the Nikol'skii inequality
		\begin{equation}\label{M3}
		\|f\|_{\infty} \leq \sqrt{Kn} \|f\|_{L_2(\Og,\mu)},\     \  \   \ \forall f\in X_n
		\end{equation}
		for some  probability measure $\mu$,
		and for any  $\e \in (0,1)$
		there is a finite set of points  $\{\xi^1,\cdots,\xi^m\}\subset \Og$  with
		$$
		m\leq C(p)\e^{-2} Kn(\log (Kn)+\log(1/\e)) (\log(1/\e)+\log \log (Kn))^2,
		$$
		which provides {the following discretization inequalities for any $f\in X_n$}
		\be\label{M4}
	{(1-\e)\|f\|_{L_p(\Og,\mu)}^p \leq \f 1m \sum_{j=1}^m |f(\xi^j)|^p \leq  (1+\e)\|f\|_{L_p(\Og,\mu)}^p.}
		\ee		
		\be\label{M4.1}
	{(1-\e)\|f\|_{L_2(\Og,\mu)}^2 \leq \f 1m \sum_{j=1}^m |f(\xi^j)|^2 \leq  (1+\e)\|f\|_{L_2(\Og,\mu)}^2.}
		\ee		
	\end{Theorem}

Several comments are in order.

{\bf Comment 1.1.}
We point out that one always has $K\ge 1$ in the Nikol'skii inequality assumed  above.

	{\bf Comment 1.2.} Historical discussions on sampling discretization  for  the cases $p=2$ and  $1\le p<2$ 	can be found in
	Subsections {\bf D.15} and    {\bf D.16} of \cite{KKLT} respectively.
	Here we only mention  that the best previously known results are the following.  It was proved in \cite{DPSTT2}  that under the Nikol'skii inequality for the pair $(2,\infty)$ we have $X_n \in \cM(m,p,\e)$ for $1\leq p<2$   provided $m\ge C(p,K,\e)n(\log n)^3$. This last estimate   on $m$ was further improved to $m\ge C(p,K,\e)n(\log n)^2$   for  $1<p<2$ in \cite{Kos}.  We also point out that 	sampling discretization for $p>2$   under the Nikol'skii inequality for the pair $(p,\infty)$ was  studied  in \cite{Kos} as well, where the results were further improved for $p>3$ in \cite{DT2}.

	{\bf Comment 1.3.}
    The Banach-Mazur distance
	between two finite dimensional spaces $X$ and $Y$ of the same dimension is defined to be
	$$
	d(X,Y):= \inf\{\|T\|\|T^{-1}\|\colon\, T\, \text{linear isomorphism from} \, X\,\text{to}\, Y\}.
	$$
	 In  Section 5  of \cite{BLM},
	the authors discuss the following important problem from functional analysis.  Given an $n$ dimensional subspace $X_n$ of $L_q(0,1)$ and $\e>0$, what is the
	smallest positive integer $N=N(X_n,q,\e)$ such that there is a subspace $Y_n$ of $\ell^{N}_q$ with $d(X_n,Y_n)\le 1+\e$? Clearly, if $X_n\in \cM(m,q,\e)$ then $N(X_n,q,C(q)\e)\le m$. On the other hand, however, results on the behavior of  $N(X_n,q,\e)$ do not seem to  imply bounds on $m$ for
	$X_n\in \cM(m,q,\e)$.   Nevertheless,   techniques
	developed for studying behavior of $N(X_n,q,\e)$ turn out to be very useful for  the  Marcinkiewicz-type discretization. A detailed  discussion of this connection can be found in Section 4.2 of \cite{KKLT}.
	 Also, we refer to  \cite{JS} for the history of this problem.
	
	{\bf Comment 1.4.} We give a brief description of the scheme of our proofs of Theorems \ref{MT1} and \ref{MT2}. Let $X_n$ be a subspace of $L_\infty$ of dimension at most $n$. First, we establish results on simultaneous discretization of the $L_2$ and $L_p$ norms. At this step
	we use  {$M \le C_p \e^{-r_1}n(\log n)^{r_2}$} points for good discretization. This step allows us to reduce the original problem of discretization to a problem in $\R^M$. Second, we use deep known results
	of Talagrand and Rudelson on the expectations of random vectors to reduce the number of points for good discretization
{by half} (approximately). Third, we iterate the second step with an appropriate stopping time and finish the proof.
	
{
{\bf Comment 1.5.} Clearly, a good discretization set $\{\xi^j\in\Omega,\, j=1,\dots,m\}$ depends on
the set $\Omega$ and on the probability measure $\mu$. For convenience, we will not indicate this fact
in our further formulations.
}

{
{\bf Comment 1.6.}
We note that Theorems \ref{MT1} and \ref{MT2} combined with
Lewis' change of density theorem (see \cite{Lew} or \cite{SZ})
imply the following two statements concerning weighted discretization
( see  the proof of Theorem 2.3 of \cite{DPSTT2} for the detailed argument).
\begin{Corollary}
		There exists a positive absolute constant  $C$  such that for any subspace $X_n$ of $\cC(\Og)$  of dimension at most $n$,
for any  $\e \in (0,1)$ and any probability measure $\mu$,
		there are  a finite set of points  $\{\xi^1,\dots,\xi^m\}\subset \Og$  with
		$$
		m\leq  C\e^{-2} n\log n,
		$$
and  a set  of nonnegative weights $\{\lambda_j\}_{j=1}^m$
		which provide the following discretization inequalities for any $f\in X_n$:
		\be
	{(1-\e)\|f\|_{L_1(\Og,\mu)} \leq  \sum_{j=1}^m \lambda_j|f(\xi^j)| \leq  (1+\e)\|f\|_{L_1(\Og,\mu)}.}
		\ee
\end{Corollary}
\begin{Corollary}
Let $1<p<2$.  There exists a positive constant  $C(p)$
        such that for any subspace $X_n$ of $\cC(\Og)$ of dimension at most $n$,
for any  $\e \in (0,1)$ and any probability measure $\mu$,
		there are  a finite set of points  $\{\xi^1,\cdots,\xi^m\}\subset \Og$  with
		$$
		m\leq C(p)\e^{-2} n(\log (n)+\log(1/\e)) (\log(1/\e)+\log \log (n))^2,
		$$
and  a set  of nonnegative weights $\{\lambda_j\}_{j=1}^m$ 
		which provide  the following discretization inequalities for any $f\in X_n$:
		\be
	{(1-\e)\|f\|_{L_p(\Og,\mu)}^p \leq  \sum_{j=1}^m \lambda_j|f(\xi^j)|^p \leq  (1+\e)\|f\|_{L_p(\Og,\mu)}^p.}
		\ee		
\end{Corollary}
}

The paper is organized as follows.  Certain iteration techniques play an important role in our proofs  of  sampling discretization results.
 In Section \ref{IL}, we prove  several technical  lemmas that will be needed in our  later applications of these  iteration techniques.

 After that, in Section \ref{p1},  we prove Theorem~\ref{MT1}, the discretization theorem of the $L_1$ norm for the finite dimensional subspaces satisfying the Nikol'skii inequality between $L_\infty$ and $L_2$ norms.  The proof combines an iteration technique with  a deep result of  Talagrand \cite{Ta90} stated in Theorem~\ref{p1T2}.
The proof also relies on a technical lemma,  Lemma \ref{p1L1},  on preliminary simultaneous discretizations.
While both  Theorem \ref{p1T2} and  Lemma \ref{p1L1}  were essentially known previously, they  were not clearly stated.   As a result,   we   include a proof of Lemma \ref{p1L1} in  Section \ref{proofL1}, and
  a  proof of Theorem \ref{p1T2}  in  Appendix I  in Section \ref{sec:7} for the sake of completeness.

Section \ref{p} is devoted to the proof of  Theorem \ref{MT2}, the discretization result  of the $L_p$ norm for $1<p<2$.
Indeed, we prove a slight improvement of Theorem \ref{MT2} in this section.  The proof uses a similar iteration technique, however, the key ingredient to the iteration  is
Theorem \ref{thm-5-1},
a  version of {Proposition 2.3 from Talagrand's paper \cite{Tal-Emb}
for $1<p<2$ with uniform  probability measure $\mu$ but weaker assumption on the $n$-dimensional space  $X_n$
(see also Theorem 16.8.2 in \cite{Ta})}.
Theorem~\ref{thm-5-1} can be deduced by modifying the proof  of Theorem 16.8.2 in  \cite{Ta}. Since the proof in  \cite{Ta} is very complicated and difficult,  to make the paper relatively self-contained, we give  a detailed proof of  Theorem \ref{thm-5-1}  in Appendix II in Section \ref{sec:8}.
Our intent there   is to present both  the result and the proof  in their simplest possible form, with a slightly weaker assumption on the space $X_n$,
so that  readers who  are not familiar with  all  the  involved technicalities (e.g.,  the majorizing measure theorem of Talagrand)
 can follow the details  easily.

	In Section \ref{df} we discuss a connection between sampling discretization of integral norms and
	frames. We formulate there (see Comment \ref{df}.2) a simple observation that the properties of a point set $\{\xi^\nu\}_{\nu=1}^m$ to provide the sampling discretization inequalities and to provide a subsystem, which is a frame, of the dictionary consisting of the Dirichlet kernels are equivalent.
	We formulate some direct corollaries of our main results on sampling discretization for construction of finite frames out of  the infinite dictionary consisting of the Dirichlet kernels.
	In Subsection \ref{df2}  we comment on the known results on sampling discretization of the uniform norm, which is closely connected with a special bilinear approximation of the Dirichlet kernels.

\section{Iteration lemmas}
\label{IL}

Iteration techniques play an important role in our proofs. In this section we present the corresponding results. The following Lemma \ref{ILL1} from \cite{NOU} was used in
proving sampling discretization results of the $L_2$ norm.

\begin{Lemma}[{\cite[Lemma 1]{NOU}}]\label{ILL1}
	Let $0<\delta<1/100$, and let $\alpha_j,\beta_j, j=0,1,\dots$, be defined inductively
	$$
\alpha_0=\beta_0=1, \quad \alpha_{j+1}:=\alpha_j\frac{1-5\sqrt{\delta/\alpha_j}}{2}, \quad \beta_{j+1}:=\beta_j\frac{1+5\sqrt{\delta/\alpha_j}}{2}.
	$$
Then there exist a positive absolute constant $C$ and a number $L\in\mathbb{N}$ such that
$$
\alpha_j\geq 100\delta\  \ \text{for}\  \  j\le L, \quad 25\delta\le \alpha_{L+1}<100\delta, \quad \beta_{L+1}<C\alpha_{L+1}.
$$
\end{Lemma}

We prove a version of Lemma \ref{ILL1}, which is convenient for us.

\begin{Lemma}\label{ILL2} Let $\de\in (0,1/4)$ and  $\theta\in(0, 1/2]$ be such that $\delta<\theta^2$. Consider the sequence
$$
\al_0=1,\qquad \al_{j+1} = \frac{1}{2}\al_j(1 -(\de/\al_j)^{1/2}),
$$
for $j=0,\dots, s$, where $s:=s(\de, \theta)\in \N_0$ is determined by the condition
\be\label{IL0}
\al_s \ge \frac{\de}{\theta^2},\qquad \al_{s+1} <\frac{\de}{\theta^2}.
\ee
Then we have the following inequalities
\be\label{IL1}
2^{-s-1}\ge\al_{s+1}\ge \frac{\de}{4\theta^2},
\ee
\be\label{IL2}
\prod_{j=0}^t ((1 + \varkappa(\de/\al_j)^{1/2}) \le \exp(c_12^{-(s-t)/2}\varkappa\theta)
\ee
and
\be\label{IL3}
\prod_{j=0}^t ((1 - \varkappa(\de/\al_j)^{1/2}) \ge \exp(-c_22^{-(s-t)/2}\varkappa\theta)
\ee
for every $0<\varkappa<{\frac{1}{2}}\theta^{-1}$, every $t\in\{0, 1, \ldots, s\}$
and  some positive absolute constants $c_1$ and $c_2$.
\end{Lemma}

\begin{proof}Consider the function $g(x):= x-(\de x)^{1/2}$, $x\in (0,1)$. Then $g'(x) = 1- (\de/x)^{1/2}/2 >0$ for $x>\de/4$. Therefore, using our assumption
	$\al_s\ge \frac{\de}{\theta^2}\ge4\delta>\de/4$ we obtain
	$$
	\al_{s+1} = \frac{1}{2}\al_s(1 -(\de/\al_s)^{1/2})=\frac{1}{2}g(\al_s)\ge
	\frac{1}{2}g(\tfrac{\de}{\theta^2})
	$$
	$$
	=
	\frac{1}{2}\Bigl(\frac{\de}{\theta^2} - \frac{\de}{\theta}\Bigr)
	=\frac{1}{2}\cdot\frac{\de}{\theta^2}(1 - \theta) \ge \frac{\de}{4\theta^2},
	$$
	which proves (\ref{IL1}).
	
	Using the trivial inequality $\al_{j+1}\le\al_j/2$, we obtain from (\ref{IL0}) that
	\be\label{IL4}
	\al_j \ge \frac{\de}{\theta^2} 2^{s-j},\qquad j=0,\dots,s.
	\ee
	This implies for $j=0,\dots,s$
	\be\label{IL5}
	1 + \varkappa(\de/\al_j)^{1/2} \le 1+\varkappa\theta2^{-(s-j)/2}
	\le \exp(\varkappa\theta2^{-(s-j)/2}).
	\ee
	Next,
	$$
	\sum_{j=0}^t 2^{-(s-j)/2}= \sum_{k=s-t}^s 2^{-k/2}
	\le \sum_{k=s-t}^\infty 2^{-k/2}
	$$
	\be\label{IL6}
	= \frac{\sqrt{2}}{\sqrt{2} - 1}\cdot2^{-(s-t)/2}
	=c_1\cdot2^{-(s-t)/2},\quad c_1:=  \frac{\sqrt{2}}{\sqrt{2} - 1},
	\ee
	and
	$$
	\prod_{j=0}^t ((1 + \varkappa(\de/\al_j)^{1/2}) \le \exp(c_12^{-(s-t)/2}\varkappa\theta),
	$$
	which proves (\ref{IL2}).
	
Since $1-x = (1+\frac{x}{1-x})^{-1}\ge \exp(-\frac{x}{1-x})\ge \exp(-2x)$
for every $x\in(0, 1/2]$
then
using (\ref{IL4}), we obtain for $0\le j\le s$
$$
1 - \varkappa(\de/\al_j)^{1/2} \ge 1- \varkappa\theta2^{-(s-j)/2} \ge
\exp(-\varkappa\theta 2^{-(s-2-j)/2}).
$$
Therefore,
$$
\prod_{j=0}^t ((1 - \varkappa(\de/\al_j)^{1/2}) \ge
\exp(-c_22^{-(s-t)/2}\varkappa\theta),
$$
which proves (\ref{IL3}).
\end{proof}

\begin{Lemma}\label{ILL3}
	Let $\{\al_j\}_{j=0}^{s+1}$ be from Lemma \ref{ILL2}, and let  $0<\varkappa<{\frac{1}{2}}\theta^{-1}$.
	Consider the sequences
	$$
	a_0=b_0=1,\quad a_{j+1} = \frac{1}{2}a_j(1 - \varkappa(\de/\al_j)^{1/2}),
	\quad
	b_{j+1} = \frac{1}{2}b_j(1 + \varkappa(\de/\al_j)^{1/2}),
	$$
	$ j=0,\dots,s$.
	Then
	\be\label{IL7}
	b_{t+1}\le \exp(c_32^{-(s-t)/2}\varkappa\theta)a_{t+1}
	\ee
	for every $t\in\{0, 1, \ldots, s\}$
	for some positive absolute constant $c_3$.
	In particular, for $\varkappa=1$ and for the sequence
	$$
	\bt_0 =1,\quad \bt_{j+1} = \frac{1}{2}\bt_j(1 +(\de/\al_j)^{1/2}),\quad j=0,\dots,s
	$$
	one has
	\be
	\bt_{t+1}\le \exp(c_32^{-(s-t)/2}\theta)\al_{t+1}
	\ee
	for every $t\in\{0, 1, \ldots, s\}$
	for some positive absolute constant $c_3$.
\end{Lemma}
\begin{proof} By the definition of $b_j$ we obtain from (\ref{IL2})
	$$
	b_{t+1} = 2^{-t}\prod_{j=0}^t ((1 + \varkappa (\de/\al_j)^{1/2})
	\le 2^{-t}\exp(c_12^{-(s-t)/2}\varkappa\theta).
	$$
	Further, using the definition of the $a_j$ and (\ref{IL3}) we get
	\be\label{IL7'}
	a_{t+1} = 2^{-t}\prod_{j=0}^t ((1 -  \varkappa(\de/\al_j)^{1/2})
	\ge 2^{-t}\exp(-c_22^{-(s-t)/2} \varkappa\theta).
	\ee
	Combining the above two inequalities, we complete the proof of (\ref{IL7}).
\end{proof}

Finally, we prove one simple inequality for a recurrent sequence.
\begin{Lemma}\label{ILL4} Let the sequence $\{m_j\}_{j=0}^\infty$ of positive numbers satisfy the following conditions:
	$$
	m_0 = M,\qquad (m_j-(m_j)^{1/2})/2 \le m_{j+1} \le m_j/2,\  \ j=0,1,\cdots
	$$
	where $M>0$ is a constant.
	Then for every integer $k\ge 0$ we have
	$$
	M-2^{k/2}M^{1/2}(\sqrt{2}-1)^{-1} \le 2^km_k \le M.
	$$
\end{Lemma}
\begin{proof} The right inequality is obvious. For the left inequality we have
	$$
	2m_k \ge m_{k-1}-(m_{k-1})^{1/2} \ge m_{k-1}-(2^{-k+1}M)^{1/2},
	$$
	which implies  $$2^k m_k\ge 2^{k-1} m_{k-1} -2^{(k-1)/2} \sqrt{M}\   \ \text{ for $k=1,2,\cdots$,}$$
	 and hence
	$$
	2^km_k \ge M- M^{1/2}2^{k/2} \sum_{j=1}^k 2^{-j/2} \ge M-2^{k/2}M^{1/2}(\sqrt{2}-1)^{-1}.
	$$
The lemma is proved.
\end{proof}

\section{The case $p=1\colon$ proof of Theorem \ref{MT1}}
\label{p1}

In this section we prove the discretization theorem of the $L_1$ norm for the finite dimensional subspaces satisfying the Nikol'skii inequality between $L_\infty$ and $L_2$ norms.  We now proceed to the detailed argument.

\begin{proof}[Proof of Theorem \ref{MT1}]

Let $\epsilon\in(0, 1/4)$
be a fixed number and let $\epsilon_0=\varkappa_1\epsilon$
and $\theta=\varkappa_2\epsilon$, where $\varkappa_1, \varkappa_2\in(0, 1)$ are positive absolute constants, which
will be specified later.

{\bf Step 1. Preliminary discretization.} We need the following lemma, which can be deduced by combining certain estimates from
\cite{DPSTT1, DPSTT2, DT}.

\begin{Lemma}\label{p1L1} Let   $1\leq p<2$ and $ 0< \epsilon_0 < 1/4$. Let $X_n$ be a subspace of $L_\infty(\Og)$ of dimension at most $n$  satisfying
	\begin{equation}\label{4-1-c}
	\|f\|_\infty \leq \sqrt{K n} \|f\|_{L_2(\Og)},\   \ \forall f\in X_n
	\end{equation}
	for some constant $K\ge 1$.    Then  there exists a finite set of points
	$x_1,\cdots, x_m\in  \Og$   with
$	m\leq C_p\epsilon_0^{-8} K n (\log (Kn))^{3}$
such that for  any $f\in X_n$,  we have
	\begin{align}
	(1-\epsilon_0) \|f\|_{L_p(\Og)}^p \leq \f 1m \sum_{j=1}^m |f(x_j)|^p  \leq (1+\epsilon_0) \|f\|_{L_p(\Og)}^p,\label{3-2a}
	\end{align}
	and
		\begin{align}
	(1-\epsilon_0) \|f\|_{L_2(\Og)}^2 \leq \f 1m \sum_{j=1}^m |f(x_j)|^2  \leq (1+\epsilon_0) \|f\|_{L_2(\Og)}^2.\label{3-3a}
	\end{align}
\end{Lemma}

   For the reader's convenience, we present  a detailed
	proof of Lemma \ref{p1L1}   in Section \ref{proofL1}.

By Lemma \ref{p1L1},  we can find
a finite set
$\Ld_M:=\{x_1,\cdots, x_M\}\subset \Og$
such that both \eqref{3-2a} and \eqref{3-3a} hold with $m=M$  for  all  $f\in X_n$.
Then  \eqref{4-1-c} implies
\begin{equation*}\label{2-3b}
\sup_{f\in X_n} \f {\|f\|_\infty} {\|f\|_{2,\Ld_M} } \leq \f {\sqrt{Kn}} {\sqrt{1-\epsilon_0}}
\leq \sqrt{2 Kn},
\end{equation*}
where  $\|f\|_{2,\Ld_M} :=(\f1M\sum_{i=1}^M |f(x_i)|^2)^{1/2}$. 	
Thus, without loss of generality, we may assume that $\Og=\Og_M=\{1,2,\cdots, M\}$
and $\mu$ is the probability measure on $\Og_M$ given by $\mu\{j\}=M^{-1}$ for $1\leq j\leq M$.

 {\bf Step 2.} We  identify  each vector in $\RR^M$ with a function on the set  $\Og_M:=\{1,\dots, M\}$.
For $I\subset \Og_M$ and $f\in \RR^M$,  we define
\[
\|f\|_{p, I} =\Bl( \f 1 {|I|} \sum_{j\in I} |f(j)|^p \Br) ^{\f1p},\   \ 1\leq p<\infty,\    \ \text{and}\   \ \|f\|_{\infty, I} :=\max_{j\in I} |f(j)|,
\]
where $|I|$ denotes the cardinality of the set $I$.
Let $\|\cdot\|_p$ denote the usual norm of $\ell_p^M$; that is,
$\|f\|_p:=M^{1/p} \|f\|_{p, \Og_M}.$ Let $B_p^M:=\{f\in \RR^M:\  \|f\|_p\leq 1\}$.
For each $I\subset \Og_M$, we denote by $R_I$ the orthogonal projection onto the space spanned by $e_i, i\in I$, where $e_1=(1,0,\cdots, 0)$, $\cdots$, $e_M=(0, \cdots, 0, 1)$ is a canonical basis of $\RR^M$. Thus,  for each $f\in \RR^M$, $(R_I f)(j)=f(j)$ for $j\in I$, and $(R_I f)(j)=0$ for $j\in\Og_M\setminus I$.
Throughout this note,   $\{ \va_i:\  \  i=1,2,\cdots\}$ denotes
a sequence  of independent Bernoulli random variables taking
values $\pm 1$ with probability $1/2$.

\begin{Theorem}\label{p1T2} Let $X_n$ be a subspace of $\RR^M$ of dimension at most $n$ satisfying
	\begin{equation}\label{2-1}
	\|f\|_{\infty}\le \sqrt{Kn} \|f\|_{2,\Og_M} = \sqrt{\f {Kn}M} \|f\|_{2},\     \  \   \ \forall f\in X_n.
	\end{equation}
	Then for $p=1$ and $p= 2$ we have
\be\label{p12}
	\EE \Bl( \sup_{f\in X_n\cap B_p^M} \Bl| \sum_{i=1}^M { \va_i}  |f(i)|^p\Br| \Br) \leq C_4
	\sqrt{\f {Kn\log n}M},
\ee
where $C_4$ is a positive absolute constant.
\end{Theorem}

Several remarks on Theorem \ref{p1T2} are in order:

\begin{itemize}

\item[\rm (i)]Theorem \ref{p1T2}  for $p=2$    was proved by Rudelson \cite[Lemma 1]{Rud1}.

 \item[\rm (ii)]For $p=1$,  Theorem \ref{p1T2} was essentially proved by   Talagrand \cite{Ta90}, while   not   explicitly  stated  there
     (see also \cite[Proposition 15.16]{LedTal} and \cite[Theorem 13]{JS}).

\end{itemize}

For completeness, we will include a  proof of Theorem \ref{p1T2} for $p=1$ in Appendix I.

The  following lemma, which is a consequence of  Theorem \ref{p1T2},  plays an important role in the proof of Theorem \ref{MT1}.

\begin{Lemma}\label{p1L2} Let  $X_n$ be   a  subspace of $\RR^M$
of dimension at most $n$ satisfying \eqref{2-1}
for some constant $K\ge 1$.  Let   $J\subset \Og_M:=\{1, 2,\cdots, M\}$. Assume that     there exist positive  constants $\al_J$, $\bt_J$    such that  for any $f\in X_n$ we have for both $p=1$ and $p=2$
\be\label{2-4}
	\al_J \| f\|^p_p \leq \|R_J f\|^p_p \leq \bt_J  \| f\|^p_p .
\ee
		Then there exists a subset $I\subset J$ with
	\begin{equation}\label{2-5}
	\f {|J|} 2 \Bl(1-\f 1 {\sqrt{|J|}} \Br) \leq |I|\leq \f {|J|} 2
	\end{equation}
such that for any $f\in X_n$ we have for both $p=1$ and $p=2$
\be\label{3-8a}
	\al_I \| f\|^p_p \leq \|R_I f\|^p_p \leq \bt_I  \| f\|^p_p
\ee
where
\be\label{p13}
\al_I:=  \f {(1-\sa_1)\al_J }2,\   \    \bt_I:=  \f {(1+\sa_1)\bt_J }2,\   \
\sa_1:= C_5\sqrt{ \f { K n\log n} {\al_J M }},
\ee
and  $C_5$ is an absolute constant.
\end{Lemma}
\begin{proof}
Without loss of generality, we may assume that  $J=\{1,2,\cdots, M_1\}$.   By \eqref{2-1}  and \eqref{2-4}, we have
\begin{equation}\label{2-6}
	\sup_{f\in X_n} \f {\|R_J f\|_\infty} {\|R_J f\|_{2}}\le \sup_{f\in X_n} \f {\|f\|_\infty} {\|R_J f\|_{2}}\leq \sqrt{\f {K_J n}{|J|}},
\end{equation}
where
$  K_J:=\f {K |J|} {\al_J M}$. Set $X_n(I) := \{R_I(f) \, :\, f\in X_n\}$.
By  Theorem \ref{p1T2} applied to  $\Og_{M_1}=J$ and $K=K_J$, we obtain for $p=1$ and $p=2$
\begin{align}
\EE \Bl( \sup_{f\in X_n(J)\cap B_p^{M_1} } \Bl|   \sum_{j=1}^{M_1}  { \va_j}  |f(j)|^p\Br| \Br) &\leq
C_4 \sqrt{ \f { K_J n\log n} {|J|}}\notag\\
&= C_4\sqrt{ \f { Kn\log n} {\al_J M}}=:\f 18\sa_1.\label{2-7}
\end{align}

Next,  consider the random set  $I:=\{i\in J:  \va_i=1\}$.  Clearly,
\[|I|=\sum_{i\in J} \f {\va_i+1}2=\f {|J|} 2 +\f12\sum_{i\in J}\va_i.\]
We use the following result from \cite{KK}.
	
	\begin{Lemma}\textnormal{\cite{KK} }\label{p1L3}
		 If $(a_1,\cdots, a_m)\in\RR^m$ and $\sum_{j=1}^m a_j^2=1$, then
		\begin{equation}\label{3-1:eq}
			\PP\Bl( |\sum_{j=1}^m a_j \va_j |\leq 1 \Br) \ge \f12.
		\end{equation}
	\end{Lemma}
Using \eqref{3-1:eq},   we have
\begin{align*}
\PP \Bl(  -\f {\sqrt{|J|}} 2\leq  |I|-\f { |J|}2\leq 0\Br)&=\PP\Bl( -\sqrt{M_1} \leq \sum_{i=1}^{M_1} \va_i \leq 0\Br) \\
&=\f 12 \PP\Bl(|\sum_{i=1}^{M_1}\va_i|\leq \sqrt{M_1}\Br)\ge \f14.
\end{align*}
This means that  with probability $\ge \f14$, we have
\begin{align}\label{2-10a}
 \f {|J|}2 \Bl(1-\f 1 {\sqrt{|J|}}\Br) \leq |I| \leq \f {|J|}2.
\end{align}
 Now combining  \eqref{2-7}  with \eqref{2-10a}, we can find a finite sequence $\{\va_j: \  1\leq j\leq M_1\}\subset \{\pm 1\}$
such that \eqref{2-10a} with  $I:=\{i\in J:  \va_i=1\}$  is satisfied, and such that  for every $f\in X_n$,
\begin{align}
&\Bl|   \sum_{i=1}^{M_1}  { \va_i}  |f(i)|^2\Br|  \leq\sa_1\|R_J f\|^2_2\    \   \text{and}\  \      \Bl|  \sum_{i=1}^{M_1}   \va_i |f(i)|  \Br|
\leq \sa_1 	\|R_J f\|_{1}.\label{4-10}
\end{align}
On the other hand, note that  for any $x\in\RR^M$ we have for $1\le q<\infty$,
\[
\Bl| \sum_{i=1}^{M_1}  \va_i |x(i)|^q \Br|=\Bl| \sum_{i=1}^{M_1}  (1+ \va_i) |x(i)|^q -\|R_J x\|_q^q \Br|=\Bl| 2\|R_I x\|_q^q -\| R_J x\|_q^q \Br|.
\]
Thus, we obtain from    \eqref{4-10} that for $p=1,2$
\begin{align}
\f {1-\sa_1} 2 \sum_{i\in J} |f(i)|^p   \leq  \|R_I f\|_p^p  \leq \f {1+\sa_1} 2 \sum_{i\in J} |f(i)|^p.\label{2-13}
\end{align}
 Combining    \eqref{2-13}  with \eqref{2-4},   we obtain the desired estimate \eqref{3-8a}.
\end{proof}

{\bf Step 3. Iteration.} For given numbers $K$, $n$, and $M$ define
$$
\de := \de(K,n,M) := C_5^2\frac{Kn\log n}{M},
$$
where $C_5$ is from (\ref{p13}). Recall that $\theta=\varkappa_2\epsilon$.
Without loss of generality we may assume that $\delta\le \theta^2$,
since otherwise we already have $M\le C_5^2\varkappa_2^{-2}{\epsilon^{-2}} Kn\log n$
and we have discretization with $\epsilon_0\le \epsilon$ on the first step.
Consider the sequence $\{\al_j\}_{j=0}^{s+1}$ from Lemma~\ref{ILL2} with $\theta=\varkappa_2\epsilon$. We now iterate applications of Lemma \ref{p1L2}. We begin with $I_0:=\{1,2,\dots,M\}$ and obtain the sequence $\{I_j\}_{j=0}^{s+1}$. Let the sequence $\{\bt_j\}_{j=0}^{s+1}$ be from Lemma~\ref{ILL3} ({i.e. we take $\varkappa=1$}). Then by Lemma \ref{p1L2} we obtain for any $f\in X_n$ and $p=1,2$
\be\label{p1j}
\al_j \|f\|_p^p \le \|R_{I_j}f\|_p^p \le \bt_j \|f\|_p^p,\quad j=0,\dots,s+1.
\ee
By Lemma \ref{ILL3} we obtain
\be\label{p13'}
\al_{s+1} \|f\|_p^p \le \|R_{I_{s+1}}f\|_p^p \le e^{c_3\theta}\al_{s+1} \|f\|_p^p.
\ee
Note that Lemma \ref{ILL3} and the trivial inequalities
$\al_{j+1}\le \al_j/2$, $\bt_{j+1}\ge \bt_j/2$ imply
\be\label{p14}
\al_{s+1} \le 2^{-s-1}
\le \bt_{s+1}\le e^{c_3\theta}\al_{s+1}.
\ee
Set $m_j:=|I_j|$, $j=0,\dots,s+1$. Then, by Lemma \ref{ILL4} we obtain
\be\label{p15}
M-2^{(s+1)/2}M^{1/2}(\sqrt{2}-1)^{-1} \le 2^{s+1}m_{s+1} \le M.
\ee
Set $m:=m_{s+1}$. Under assumption $\de \ge 4(\sqrt{2}-1)^{-2}M^{-1}$, which we certainly can impose without loss of generality, we obtain from (\ref{p14}) and (\ref{IL1}) that
$$
4\theta^2 \cdot 2^{-s-1}
\ge
4\theta^2\cdot \al_{s+1}
\ge
4\theta^2\cdot \frac{\de}{4\theta^2}=\delta\ge4(\sqrt{2}-1)^{-2}M^{-1}.
$$
Thus, $2^{(s+1)/2}M^{-1/2}\le \theta(\sqrt{2}-1)$ and,
combining this bound with (\ref{p15}) and with (\ref{p14}),
we obtain
\be\label{p16}
\alpha_{s+1}(1-\theta)M\le 2^{-s-1}(1-\theta)M
\le m \le 2^{-s-1}M\le e^{c_3\theta}\alpha_{s+1}M.
\ee

Finally, (\ref{p13'}) and (\ref{p16}) imply
$$
e^{-c_3\theta}\frac{1}{M}\|f\|_p^p \le
\frac{1}{m}\sum_{k\in I_{s+1}} |f(k)|^p \le (1-\theta)^{-1}e^{c_3\theta}\frac{1}{M}\|f\|_p^p
$$
with
$$
m=|I_{s+1}| \le 2^{-s-1}M\le e^{c_3\theta}\al_{s+1}M \le e^{c_3\theta}\theta^{-2}\de M
$$
$$
=C_5^2 \varkappa_2^{-2}e^{c_3\varkappa_2\epsilon}
\epsilon^{-2}Kn\log n \le C_8\epsilon^{-2} Kn\log (n).
$$
We now choose $\varkappa_1$ and $\varkappa_2$ such that
$$
1-\epsilon\le (1-\epsilon_0)e^{-c_3\theta} \hbox{ and }
(1+\epsilon_0)(1-\theta)^{-1}e^{c_3\theta}\le 1+\epsilon.
$$
Theorem is proved.
\end{proof}

\section{The case $1<p<2\colon$ proof of Theorem \ref{MT2}}
\label{p}

The main purpose of this section is to prove Theorem \ref{MT2}.
The argument  follows along the lines of Section \ref{p1}.
Indeed, we will prove  a slight improvement of Theorem \ref{MT2}.

\begin{Theorem}\label{T1p} Let  $X_n$ be a subspace of $\cC(\Og)$ of dimension at most $n$
	satisfying the Nikol'skii inequality
	\begin{equation}\label{M3}
	\|f\|_{\infty} \leq \sqrt{Kn} \|f\|_{L_2(\Og)},\     \  \   \ \forall f\in X_n
	\end{equation}
	for some constant  $K\ge 1$. Then for any 	$1<p<2$,  there exists a positive constant $C_1(p)$  depending only on $p$  such that for any $\ep \in (0,1/4)$,
		there is a finite set of points  $\{\xi^1,\cdots,\xi^m\}\subset \Og$  with
	$$
	m\leq C_{1}(p)\va^{-2} Kn\Bl(\log (Kn)+\log(1/\va)\Br) \Bl(\log(1/\va)+\log \log (Kn)\Br)^2,
	$$
 which provides the discretization inequalities for both  the $L_p$ norm and  the $L_2$ norm:
		\begin{align*}
		(1-\ep)\|f\|_{L_p(\Og)}^p \leq \f 1m \sum_{j=1}^m |f(x_j)|^p \leq  (1+\ep)\|f\|_{L_p(\Og)}^p,\   \ \forall f\in X_n,
		\end{align*}
$$
(1-\ep_n)\|f\|_{L_2(\Og)}^2  \leq \f 1m \sum_{j=1}^m |f(x_j)|^2 \leq  (1+\ep_n)\|f\|_{L_2(\Og)}^2,\   \ \forall f\in X_n,
$$
with $\ep_n=\ep\cdot  (\log \log (2Kn))^{-1}$.	
		
\end{Theorem}

\begin{proof}

Let $\epsilon\in(0, 1/4)$
be a fixed number, and set
$$
\epsilon_0 := \frac{\varkappa_1\epsilon}{\log\log({4}Kn)},
$$
where $\varkappa_1\in (0, 1)$ is an absolute constant to be specified later. \\

	 {\bf Step 1p. } This step is the same as in the proof of Theorem \ref{MT1}.
We use Lemma \ref{p1L1}
to   obtain
a finite set $\Ld_M:=\{x_1,\cdots, x_M\}\subset \Og$
with $M\leq C_p\epsilon_0^{-8} K n (\log (Kn))^3 $
such that both \eqref{3-2a} and \eqref{3-3a} hold
with $m=M$  for  all  $f\in X_n$. Without loss of generality, we may also assume that
	\begin{equation}\label{4-1-2022}
			M\ge  C_9\va^{-2} Kn(\log (Kn)+\log(1/\va)) (\log(1/\va)+\log \log (Kn))^2,
	\end{equation}
since otherwise there's nothing to prove.

{\bf Step 2p.} It is similar to Step 2 of the proof of Theorem \ref{MT1}. Instead of Theorem \ref{p1T2} we use the following result.

\begin{Theorem}\label{thm-5-1}
	Let $X_n$ be a  subspace of $\RR^M$ of dimension at most $n$ satisfying
	\begin{equation}\label{5-1-Ta}
	\|f\|_{\infty}\le \sqrt{Kn} \|f\|_{2,\Og_M} = \sqrt{\f {Kn}M} \|f\|_{2},\     \  \   \ \forall f\in X_n
	\end{equation}
	for some constant  $K\ge 1$.
	Then for $p\in (1,2)$, we have
	\begin{align}
	\EE \Bl( \sup_{f\in X_n\cap B_p^M } \Bl| \sum_{i=1}^M |f(i)|^p\va_i  \Br|\Br) \leq C(p) \sqrt{\f {Kn\log M} M} \log (\f M{Kn}+2).\label{3-1-Ta}
	\end{align}
\end{Theorem}


For $1<p<2$,    Talagrand \cite[Theorem 16.8.2]{Ta} proved  a more general result   for a probability measure $\mu$ on $\RR^M$ satisfying $\mu\{j\}\leq \f 2M$ for $1\leq j\leq M$ under a stronger assumption on the space  $X_n$:
\[ \f 1n \sum_{j=1}^n \vi_j(i)^2 =1,\   \ i=1,2,\cdots, M,\]
where $\{\vi_j\}_{j=1}^n$ is an orthonormal basis of $(X_n, \|\cdot\|_{L_2(\mu)})$.  The proof of Talagrand \cite[Theorem 16.8.2]{Ta} is very difficult, but can be modified and slightly simplified to obtain Theorem \ref{thm-5-1}.
 For completeness, we will present  a relatively self-contained  proof of  Theorem \ref{thm-5-1}  in Appendix II.

 In the same way as Lemma \ref{p1L2} has been proven in Section \ref{p1} we can prove the following result.

\begin{Lemma}\label{L1p} Let  $X_n$ be   a subspace of $\RR^M$ of dimension at most $n$ satisfying \eqref{5-1-Ta}
	for some constant $K\ge 1$.  Let  $1<p<2$ and   $J\subset \Og_M:=\{1, 2,\cdots, M\}$. Assume that     there exist positive  constants $\al_J$, $\bt_J$, $a_J$, $b_J$    such that  for any $f\in X_n$,
	\be\label{5p}
	\al_J \| f\|^2_2 \leq \|R_J f\|^2_2 \leq \bt_J  \| f\|^2_2\   \   \ \text{and}\   \ 		a_J \| f\|^p_p \leq \|R_J f\|^p_p \leq b_J  \| f\|^p_p.
\ee
		Then there exists a subset $I\subset J$ with
	\begin{equation}\label{6p}
	\f {|J|} 2 \Bl(1-\f 1 {\sqrt{|J|}} \Br) \leq |I|\leq \f {|J|} 2
	\end{equation}
such that
	\begin{align}
\al_I \| f\|^2_2 \leq \|R_I f\|^2_2 \leq \bt_I  \| f\|^2_2\   \   \ \text{and}\   \ 		a_I \| f\|^p_p \leq \|R_I f\|^p_p \leq b_I  \| f\|^p_p,\label{7p}
\end{align}
where
\begin{align*}
\al_I:&=  \f {(1-\sa_1)\al_J }2,\   \    \bt_I:=  \f {(1+\sa_1)\bt_J }2,\   \ \sa_1:= C_5\sqrt{ \f { K n\log n} {\al_J M }},\\
a_I:&=  \f {(1-\sa_2)a_J }2 ,\   \
 b_I:=  \f {(1+\sa_2)b_J }2,
\end{align*}
$$
 \sa_2:= C_p\sqrt{ \f { K n\log |J|} {\al_J M }} \log \Bl(2+\f {M}{Kn}\Br),
 $$	
and  $C_p$ is a constant  depending  only on $p$.
\end{Lemma}

We point out that
$$
\sa_1\le
C\sqrt{ \f { K n\log M} {\al_J M }}:=\sa_1'\quad
\sa_2
\le
C
\sqrt{ \f { K n\log M} {\al_J M }} \log \Bl( 2+\f {M}{Kn}\Br):=\sa_2',
$$
where $C:=\max\{C_5, C_p\}$.
Thus, we can use $\sa_1'$ and $\sa_2'$ in place of
$\sa_1$ and $\sa_2$ in the lemma above.

{\bf Step 3p.}
As in the proof of Theorem \ref{MT1},  we iterate applications of Lemma \ref{L1p}
and obtain sequences $\{\al_j\}_{j=0}^{s+1}$, $\{\bt_j\}_{j=0}^{s+1}$,
$\{a_j\}_{j=0}^{s+1}$, $\{b_j\}_{j=0}^{s+1}$ as defined  in Lemma~\ref{ILL2} and
Lemma~\ref{ILL3} with
$$
\de:= C^2\frac{Kn\log M}{M},\quad \varkappa:= \log\Bl( 2+\f {M}{Kn}\Br)\   \ \text{and}\  \
\theta := \frac{\varkappa_2\epsilon}{\varkappa},
$$
where $\varkappa_2\in (0, 1)$ is an absolute constant to be specified later.
{Recalling that}  $$ M\leq C_p\epsilon_0^{-8} K n (\log (Kn))^3\   \ \text{and}\  \
\epsilon_0 = \frac{\varkappa_1\epsilon}{\log\log(4Kn)},
$$   we may choose   the constant $C_9$ in \eqref{4-1-2022} sufficiently large so that the parameters $\delta$, $\theta$ and $\varkappa$ satisfy all  the conditions of Lemma~\ref{ILL2} and
Lemma~\ref{ILL3}; that is,
$$  \delta\in (0, \f14),\   \  \theta\in (0, \f12),\   \ \delta<\theta^2,\  \ 0<\varkappa<{\frac{1}{2}}\theta^{-1}.$$
Thus, we have
$$
a_{s+1} \|f\|_p^p \le \|R_{I_{s+1}}f\|_p^p \le b_{s+1} \|f\|_p^p
\le \exp(c_3\varkappa\theta)a_{s+1}\|f\|_p^p;
$$
$$
\al_{s+1} \|f\|_2^2 \le \|R_{I_{s+1}}f\|_2^2 \le \bt_{s+1} \|f\|_2^2
\le \exp(c_3\theta)\al_{s+1}\|f\|_2^2.
$$

As in the proof of Theorem \ref{MT1}, for $m=m_{s+1} = |I_{s+1}|$ we have
$$
2^{-s-1}(1-\varkappa\theta)M \le2^{-s-1}(1-\theta)M \le m \le 2^{-s-1}M.
$$
Since
$$
\al_{s+1} \le 2^{-s-1}
\le \bt_{s+1}\le e^{c_3\theta}\al_{s+1};\quad
a_{s+1}\le 2^{-s-1}\le b_{s+1}\le e^{c_3\varkappa\theta}a_{s+1},
$$
we get
$$
e^{-c_3\varkappa\theta}\frac{1}{M}\|f\|_p^p \le
\frac{1}{m}\sum_{k\in I_{s+1}} |f(k)|^p \le
(1-\varkappa\theta)^{-1}e^{c_3\varkappa\theta}\frac{1}{M}\|f\|_p^p
$$
and
$$
e^{-c_3\theta}\frac{1}{M}\|f\|_2^2 \le
\frac{1}{m}\sum_{k\in I_{s+1}} |f(k)|^2 \le
(1-\theta)^{-1}e^{c_3\theta}\frac{1}{M}\|f\|_2^2
$$
with
$$
m=|I_{s+1}| \le 2^{-s-1}M\le e^{c_3\theta}\al_{s+1}M \le e^{c_3\theta}\theta^{-2}\de M.
$$
We now choose $\varkappa_1$ and $\varkappa_2$ such that
$$
1-\epsilon\le (1-\varkappa_1\epsilon)e^{-c_3\varkappa_2\epsilon} \hbox{ and }
(1+\varkappa_1\epsilon)(1-\varkappa_2\epsilon)^{-1}e^{c_3\varkappa_2\epsilon}
\le 1+\epsilon.
$$
Since $\log M\le c_1(\log\epsilon^{-1} + \log (Kn))$ and
$\varkappa\le c_2(\log\epsilon^{-1} + \log\log(Kn))$
we get the desired result. Theorem is proved.
\end{proof}

\section{Discretization and frames}
\label{df}

In this section we discuss finite-dimensional subspaces $X_n$ of the space $\cC(\Og)$ defined
on a compact set $\Og\subset \R^d$ equipped with a probability measure $\mu$. For convenience we only consider the case of real functions.


{\bf Dirichlet kernel.} For an orthonormal system $\cU_n:=\{u_j(\bx)\}_{j=1}^n$  on $(\Og,\mu)$
we define the Dirichlet kernel as follows
$$
\cD_n(\cU_n,\bx,\by) := \sum_{j=1}^n u_j(\bx)u_j(\by).
$$
Here is a very simple claim that the Dirichlet kernel $\cD_n(\cU_n,\bx,\by)$ does not depend on the orthonormal basis of a given subspace $X_n$.

\begin{Proposition}\label{dfP1} For any two orthonormal bases $\cU_n$ and $\cV_n$ of
	a given subspace $X_n$ we have
	$$
	\cD_n(\cU_n,\bx,\by) = \cD_n(\cV_n,\bx,\by).
	$$
\end{Proposition}
\begin{proof} For $\bx\in\Og$ denote the column vectors $\bu(\bx) := (u_1(\bx),\dots,u_n(\bx))^T$, $\bv(\bx) := (v_1(\bx),\dots,v_n(\bx))^T$. Then there exists an orthogonal matrix $O$ such that for all $\bx\in\Og$ we have $\bv(\bx)= O\bu(\bx)$ and, therefore,
	$\bv(\bx)^T= \bu(\bx)^TO^T$. This implies that
	$$
	\cD_n(\cV_n,\bx,\by) = \bv(\bx)^T\bv(\by) = \bu(\bx)^TO^TO\bu(\by) =  \bu(\bx)^T \bu(\by) =
	\cD_n(\cU_n,\bx,\by).
	$$
\end{proof}

Proposition \ref{dfP1} shows that the Dirichlet kernel $\cD_n(\cU_n,\bx,\by)$ with $\cU_n$ being an orthonormal basis of $X_n$ can be seen as a characteristic of the subspace $X_n$.
Denote
$$
\cD(X_n,\bx,\by) := \cD_n(\cU_n,\bx,\by).
$$
Consider the system $\cD :=\{\cD(X_n,\bx,\by)\}_{\bx\in\Og}$ as a dictionary (not normalized) of
functions on $\by$ in the subspace $X_n$.

We recall the definition of the $p$-frame of the subspace $X_n$, $1\le p<\infty$ (see \cite{AFR}).

\begin{Definition}\label{dfD1} The system $\Psi:=\{\psi_j\}_{j=1}^m$ is said to be a $p$-frame
	of $X_n$ with positive constants $A$ and $B$ if for any $f\in X_n$ we have
	$$
	A\|f\|_{L_p(\Og,\mu)}^p \le \sum_{j=1}^m |\<f,\psi_j\>|^p \le B\|f\|_{L_p(\Og,\mu)}^p.
	$$
\end{Definition}

Using a well known property of the Dirichlet kernel: For any $f\in X_n$ we have
$$
f(\bx) = \int_\Og \cD(X_n,\bx,\by)f(\by)d\mu(\by),
$$
we derive from results of Sections \ref{p1} and \ref{p} the following corollaries.

\begin{Theorem}\label{dfT1} There exist three positive absolute constants $C_i$, $i=1,2,3$ such that for any $n$-dimensional subspace $X_n$ of  $\cC(\Og)$
	satisfying the Nikol'skii inequality
	\begin{equation}\label{df1}
	\|f\|_{\infty} \leq \sqrt{Kn} \|f\|_{L_2(\Og)},\     \  \   \ \forall f\in X_n
	\end{equation}
	for some constant $K\ge 1$,
	there is a finite set of points  $\xi^1,\cdots,\xi^m\in \Og$  with $m\leq C_{1} Kn\log n $
	such that the finite subdictionary $\Psi :=\{\cD(X_n,\xi^\nu,\by)\}_{\nu=1}^m$ of dictionary $\cD$ forms
	$p$-frames of the $X_n$ with constants $C_2m$ and $C_3m$ for $p=1$ and $p=2$.
\end{Theorem}

\begin{Theorem}\label{dfT2} Let $1<p<2$. There exist three positive constants $C_i'$, $i=1,2,3$, ($C_1'$ may depend on $p$ and $C_2'$ and $C_3'$ are absolute constants) such that for any $n$-dimensional subspace $X_n$ of $\cC(\Og)$
	satisfying the Nikol'skii inequality
	\begin{equation}\label{df2}
	\|f\|_{\infty} \leq \sqrt{Kn} \|f\|_{L_2(\Og)},\     \  \   \ \forall f\in X_n
	\end{equation}
	there is a finite set of points  $\xi^1,\cdots,\xi^m\in \Og$  with
	$$
	m\leq C_{1}' Kn\log (Kn) (\log \log (Kn))^2
	$$
	such that the finite subdictionary $\Psi :=\{\cD(X_n,\xi^\nu,\by)\}_{\nu=1}^m$ of dictionary $\cD$ forms a
	$p$-frame of the $X_n$ with constants $C'_2m$ and $C'_3m$  and forms a
	$2$-frame of the $X_n$ with constants $(1-\ep_n)m$ and $(1+\ep_n)m$,
	where
	$$\ep_n \asymp (\log \log (Kn))^{-1}.$$	
	
\end{Theorem}

{\bf Comment \ref{df}.1.} Theorem \ref{dfT1} is a corollary of Theorem \ref{MT1}. If instead of Theorem \ref{MT1} we use the known results from \cite{LT} on the sampling discretization of the $L_2$ norm, then we obtain a version of Theorem \ref{dfT1} about a $2$-frame: Under condition (\ref{df1}) there exists a
finite subdictionary $\Psi :=\{\cD(X_n,\xi^\nu,\by)\}_{\nu=1}^m$ of dictionary $\cD$, which forms a
$2$-frame of the $X_n$ with constants $C_2m$ and $C_3m$ with $m\le C_1Kn$.

{\bf Comment  \ref{df}.2.} We pointed out  that Theorems \ref{dfT1} and \ref{dfT2} about frames are corollaries of Theorems \ref{MT1} and \ref{T1p} on the sampling discretization. Actually, these problems are equivalent. Namely, the following two statements are equivalent.

{\it Statement 1.} The set of points $\{\xi^\nu\}_{\nu=1}^m$ is such that for all $f\in X_n$ we have the sampling discretization inequalities
$$
A\|f\|_p^p \le \frac{1}{m} \sum_{\nu=1}^m |f(\xi^\nu)|^p \le B\|f\|_p^p.
$$

{\it Statement 2.} The set of points $\{\xi^\nu\}_{\nu=1}^m$ is such that the subsystem
$\Psi :=\{\cD(X_n,\xi^\nu,\by)\}_{\nu=1}^m$ of the system $\cD$ forms a
$p$-frame of the $X_n$ with constants $Am$ and $Bm$.

{\bf Comment  \ref{df}.3.} We have discussed above a connection between sampling discretization
with equal weights of the $L_p$ norm of functions from a subspace $X_n$ and the $p$-frame properties of a subsystem of the system generated by
the corresponding Dirichlet kernel. In the sampling discretization theory we study weighted discretization
along with discretization with equal weights. This motivates us to introduce a concept of
$(\La,p)$-frame. Here is a rigorous definition.

\begin{Definition}\label{dfD2} Let $\La:=\{\la_j\}_{j=1}^m$ and $1\le p<\infty$. The system $\Psi:=\{\psi_j\}_{j=1}^m$ is said to be a $(\La,p)$-frame
	of $X_n$ with positive constants $A$ and $B$ if for any $f\in X_n$ we have
	$$
	A\|f\|_{L_p(\Og,\mu)}^p \le \sum_{j=1}^m \la_j |\<f,\psi_j\>|^p \le B\|f\|_{L_p(\Og,\mu)}^p.
	$$
\end{Definition}

Then the following two statements are equivalent.

{\it Statement 1w.} The sets of points $\{\xi^\nu\}_{\nu=1}^m$ and weights $\La:=\{\la_\nu\}_{\nu=1}^m$ are such that for all $f\in X_n$ we have the weighted sampling discretization inequalities
$$
A\|f\|_p^p \le   \sum_{\nu=1}^m \la_\nu |f(\xi^\nu)|^p \le B\|f\|_p^p.
$$

{\it Statement 2w.} The sets of points $\{\xi^\nu\}_{\nu=1}^m$ and weights $\La:=\{\la_\nu\}_{\nu=1}^m$ are such that the subsystem
$\Psi :=\{\cD(X_n,\xi^\nu,\by)\}_{\nu=1}^m$ of the system $\cD$ forms a
$(\La,p)$-frame of the $X_n$ with constants $A$ and $B$.

{\bf Comment  \ref{df}.4.} Let $\cU_n:=\{u_j(\bx)\}_{j=1}^N$ be an orthonormal  basis of $X_n$ on $(\Og,\mu)$.
It is well known and easy to see that
$$
\sup_{f\in  X_n, \|f\|_2\le 1} |f(\bx)| = \left(\sum_{j=1}^n u_j(\bx)^2\right)^{1/2} =: w(\bx).
$$
The function $w(\bx)^{-1}$ is known as the Christoffel function of the subspace $X_n$.
Clearly,
$$
w(\bx)^2 = \cD(X_n,\bx,\bx).
$$
Thus, the conditions (\ref{df1}) and (\ref{df2}) in Theorems \ref{dfT1} and \ref{dfT2} can be
formulated as
$$
\cD(X_n,\bx,\bx) \le Kn,\quad \bx\in \Og.
$$
Note, that the condition $w(\bx)^2 \le nt^2$, $\bx\in\Og$, is known in the sampling discretization
theory under the name Condition E (see, for instance, \cite{KKLT}).

{
{\bf Comment  \ref{df}.5.} It is known (see \cite{KKT}) that sampling discretization of the uniform norm of functions from $X_n$ is connected to a special type of bilinear approximation of the Dirichlet kernel $\cD(X_n,\bx,\by)$ of this subspace. We refer the reader to \cite{KKT} for a detailed discussion of the corresponding results.
}

\section{Proof of Lemma \ref{p1L1}}
\label{proofL1}

For the proof of Lemma \ref{p1L1}, we  need   several technical lemmas from
\cite {DPSTT1, DPSTT2, DT}.
First,  we need a  conditional result from \cite{DT},  which allows us to estimate  the number of  points needed for  the sampling discretization in terms of  an integral of $\va$-entropy.  Recall  that  given a positive number $\e$,  	the $\va$-entropy $\mathcal{H}_\va (A, X)$ of the compact set $A$  in  a Banach space $(X, \|\cdot\|_X)$   is  defined as
$\log_2 N_\e(A,X)$, where
\[ N_\va (A,X):=\min\Bl\{ n\in\NN:\  \ \exists\; g^1,\ldots, g^n\in A, \   \sup_{f\in A} \min_{1\leq j\leq n} \|f-g^j\|_X\leq \va\Br\}.\]

\begin{Lemma}\label{prL1}\textnormal{\cite[Theorem 5.1]{DT}}
	Let $\CW$  be a set of uniformly bounded functions on $\Og$ with
	$$
	1\leq 	R:=\sup_{f\in \CW}\sup_{x\in\Og}  |f(x)|<\infty.
	$$
	Assume that  $\cH_{ t}(\CW,L_\infty)<\infty$ for every $t>0$, and
	\begin{equation}\label{5.1b}
	(\lambda\cdot \CW)\cap B_{L_p} \subset \CW\subset B_{L_p},\   \   \ \forall \lambda>0
	\end{equation}
	for some $1\leq p<\infty$, where $B_{L_p}:=\{f\in L_p(\Og): \|f\|_p\leq 1\}  $.
	Then there exist positive constants $C_p, c_p$ depending only on $p$   such that for any  $\va\in (0, 1)$ and  any
	integer
	\begin{equation}\label{5-2-c}
	m\ge  C_p 	  \va^{-5 }\left(\int_{10^{-1}\va^{1/p}} ^{R} u^{\frac  p2-1}    \Bigl(\int_{  u}^{ R }\frac {\cH_{c_p \va t}(\CW,L_\infty)}t\, dt\Bigr)^{\frac 12} du\right)^2,
	\end{equation}
	there exist  $m$ points $x_1,\cdots, x_m\in \Og$  such that for all $f\in  \CW$,
	\begin{equation}\label{5.3b}
	(1-\va) \|f\|_p^p \leq \frac  1m \sum_{j=1}^m |f(x_j)|^p\leq (1+\va) \|f\|_p^p.
	\end{equation}
\end{Lemma}
\begin{Remark}\label{prR1} The proof  in \cite{DT} actually yields the following result.  Let $x_1, \cdots, x_m$ be independent random points  identically distributed according to  the probability measure $\mu$ on $\Og$. Then  under the  conditions of  Lemma~\ref{prL1},  the inequalities  \eqref{5.3b}  hold for all $f\in\CW$   with probability $>\f34$.
	
\end{Remark}

%

We will also need  the  following estimates of  entropy numbers from \cite{DPSTT2}. 
\begin{Lemma} \label{prL2}\textnormal{\cite[Theorem 2.1]{DPSTT2}}	Assume that  $X_n$ is  an $n$-dimensional subspace of $L_\infty(\Og)$ satisfying the following two conditions:
	\begin{enumerate}
		\item [{\bf \textup{(i)} }] There exists a constant $K_1>1$ such that
		\begin{equation}\label{4-1}
		\|f\|_\infty \leq (K_1 n)^{\f12}\|f\|_2,\   \ \forall f\in X_n.
		\end{equation}

		\item [{\bf \textup{(ii)}}] There exists a constant $K_2>1$  such that for $q_n:= \log n$,
		\begin{equation}\label{2-2}
		\|f\|_\infty \leq K_2 \|f\|_{q_n
		},\   \ \forall f\in X_n.
		\end{equation}
	\end{enumerate}
	Then for  each $1\leq p\leq 2$,  there exists a constant $C_p>0$ depending only on $p$ such that
	\begin{equation}\label{4-2-3b}
	\cH_{t} (X_n^p; L_\infty) \leq C_p K_1K_2^2 \f { n \log n} {t^p},\    \   \  \forall t>0,
	\end{equation}
	where
	$
	X^p_n := \{f\in X_n:\, \|f\|_p \le 1\}.
	$
\end{Lemma}

Note that  Lemma  \ref{p1L1}  can be deduced directly from Lemma \ref{prL1}, Remark \ref{prR1} and Lemma \ref{prL2} under the additional assumption 	
\eqref{2-2}.   To drop the extra condition \eqref{2-2}, we need the following lemma proved in \cite{DPSTT1}.

\begin{Lemma}\label{prL3} 	\textnormal{\cite[Lemma 4.3] {DPSTT1}}  Let   $1\leq p<\infty$ and $ 0< \va < 1/4$. Let $X_n$ be an $n$-dimensional subspace of $L_\infty(\Og)$ satisfying
	\begin{equation*}
	\|f\|_\infty \leq
	(Kn)^{\frac 1 p} \|f\|_{L_p(\Og)},\    \    \ \forall f\in X_n
	\end{equation*}
	for some constant $K\ge 1$.
	Let $\{\xi_j\}_{j=1}^\infty$ be a sequence of independent random points  identically distributed in accordance with $\mu$ on the set $\Og$.
	Then   there exists an absolute   constant  $C>0$   such that for   any integer   \begin{equation}
	m\ge C K \va^{-2}\Bl( \log\frac 2\va\Br) n^{2}\log n,\label{4-3-a}
	\end{equation}   with probability
	$ \ge 1-m^{-n/\log K}$, one has
	\begin{align}
	(1-\va) \|f\|_p^p\leq  \frac 1m \sum_{j=1}^m| f(\xi_j)|^p \leq (1+\va)\|f\|_p^p,\  \ \forall f\in X_n.
	\end{align}
	
\end{Lemma}

\begin{proof}[Proof of Lemma \ref{p1L1}] 	Without loss of generality, we may assume that $\log K\le  C \log n$ for some absolute constant $C>1$ since otherwise we may use $Kn$ to replace $n$, considering $X_n$ as a subspace of dimension at most $Kn$. We may also  assume that $0<\va<\va_0$ and $n\ge n_0$, where $\va_0\in (0, 1)$ is a sufficiently small absolute constant, and $n_0>1$ is a sufficiently large absolute constant. 	
	For     $\bz=(z_1,\ldots, z_{m})\in\Og^m$, define the  operator
	$S_{m, \mathbf{z}} : X_n \to \RR^{m}$ by
	$ S_{m, \mathbf{z}} f 
	=(f(z_1), \ldots, f(z_{m})).$
	Note that \eqref{4-1-c} implies that (see, for instance, \cite{DPSTT2})
	$$\|f\|_\infty \leq (Kn)^{\f1p} \|f\|_p,\   \ \forall f \in X_n,\   \  1\leq p\leq 2.$$
	By  Lemma \ref{prL3},   there exists  a vector   $\bz=(z_1,\ldots, z_{m_1})\in\Og^{m_1}$
	with
	$$C\va^{-2}(\log \f 1 \va) K n^2 \log n \leq m_1\leq C^2\va^{-2}(\log \f 1 \va) K n^2\log n$$
	such that
	\begin{equation}\label{6-7}
	\Bl|\|f\|_{L_p(\Og)} - \|S_{m_1, \bz} f\|_{p,\Og_{m_1}} \Br|\leq \f \va4\|f\|_p,\   \ \forall f\in X_n^p,\
	\end{equation}
	and
	\begin{equation}\label{4-1-3}
	\Bl|\|f\|_{L_2(\Og)} - \|S_{m_1, \bz} f\|_{2,\Og_{m_1}} \Br|\leq \f \va4\|f\|_2,\   \ \forall f\in X_n^2,\
	\end{equation}
	where $\Og_{m_1}:=\{z_1,\cdots, z_{m_1}\}$.
	
	
	Now consider the $n$-dimensional subspace    $\wt{X}_n :=S_{m_1,\bz} (X_n)$  of $L_\infty(\Og_{m_1})$. Using  \eqref{4-1} and \eqref{4-1-3}, we have that   for any $f\in X_n$,
	\begin{equation}\label{6-7-5}\|S_{m_1,\bz} f \|_{\infty, \Og_{m_1}} \leq \sup_{x\in\Og} |f(x)| \leq (Kn)^{\f12} \|f\|_2 \leq (2K n)^{\f12}\|S_{m_1,\bz} f\|_{2,\Og_{m_1}}.
	\end{equation}
	Furthermore, since $\log K \leq C \log n$, we have $$\log m_1\leq    3\log \va^{-1}+  C\log (\sqrt{K}n))\leq 3\log \va^{-1}+C \log n,$$ which in turn implies that  for  any $f:\Og_{m_1}\to\RR$,
	\[ \|f\|_{\infty, \Og_{m_1}}\leq  m_1^{\f 1 {q_n}} \|f\|_{q_n,\Og_{m_1}}\leq C \va^{-\f14}\|f\|_{q_n,\Og_{m_1}}.\]
	Thus,   the $n$-dimensional subspace  $\wt{X}_n$ of $L_\infty(\Og_{m_1})$ satisfies both the conditions \eqref{4-1} and \eqref{2-2} with $K_1=2K$ and $K_2=C \va^{-\f14}$. By
	Lemma 	\ref{prL2} (applied with a discrete  measure),
	we then obtain
	\begin{equation}
	\cH_{t} (\wt X_n^p; \|\cdot\|_{\infty, \Og_{m_1}})  \leq C_p K \va^{-\f12}  \f { n \log n} {t^p},\    \   \  \forall t>0,
	\end{equation}
	where
	$$
	\wt X^p_n := \{f\in \wt X_n:\, \|f\|_{p,\Og_{m_1}}\le 1\}.
	$$
	This implies that for  $R=(2Kn)^{1/p}$ and any $1\leq p\leq 2$,
	\begin{align*}
	\va^{-5 }\left(\int_{10^{-1}\va^{1/p}} ^{R} u^{\frac  p2-1}    \Bigl(\int_{  u}^{ R }\frac {\cH_{c_p \va t}(\wt X_n^p,\|\cdot\|_{\infty, \Og_{m_1}})}t\, dt\Bigr)^{\frac 12} du\right)^2\leq C \va^{-8} Kn \log^3 n.
	\end{align*}
	Thus, applying Lemma \ref{prL1} and Remark \ref{prR1}   to the subspace $\wt X_n$ of $L_p(\Og_{m_1})$,
	we  get  a  subset  $\Ld \subset\{1,2,\ldots, m_1\}$ with $|\Ld| \leq C_p\va^{-8}  K n\log^3 n$ such that the inequalities
	$$ \Bl| \f 1{| \Ld|} \sum_{j\in\Ld} |S_{m_1,\bz} f(j)|^q  -\|S_{m_1, \bz} f\|^q_{q,\Og_{m_1}}\Br|\leq \f \va 4\|S_{m_1, \bz} f\|^q_{q,\Og_{m_1}},\   \    f\in X_n $$
	hold simultaneously for $q=p$ and $q=2$.
	This together with  \eqref{6-7} and \eqref{4-1-3} implies that for $q=p$ and $q=2$,
	$$(1-\va) \|f\|_q^q \leq \f 1 {| \Ld|} \sum_{j\in\Ld} |f(z_j)|^q \leq (1+\va) \|f\|_q^q,\  \ \forall f\in X_n.$$	
The theorem is proved.	
\end{proof}
\begin{Remark}
	We  point out that Lemma \ref{p1L1} for  $1<p<2$ also follows directly from  a result of Rudelson (Theorem \ref{T-Rud}) and
	a  result from  \cite{Kos} (see Theorem \ref{T-Kos}).
	
		\begin{Theorem}[\cite{Rud1},  see also Corollary 4.1 of \cite{KKLT}]\label{T-Rud}
			Let $\mu$ be a probability Borel measure
			on a compact set $\Omega$.
			There is a constant $C$ such that,
			if
			$X_{n}$ is an $n$-dimensional subspace of $L_\infty(\mu)$
			such that
			$$
			\|f\|_\infty\le \sqrt{Kn}\|f\|_2\quad \forall f\in X_{n},
			$$
			then
			$$
			\mathbb{E}\sup\limits_{f\in B_2(X_{n})}
			\Bigl|\frac{1}{m}\sum\limits_{j=1}^m|f(x_j)|^2 - \|f\|_2^2\Bigr|
			\le C\Bigl(\frac{Kn\log n}{m} + \sqrt{\frac{Kn\log n}{m}}\Bigr),
			$$
			where $B_2(X_{n}):=\{f\in X_{n}\colon \|f\|_2\le 1\}$.
		\end{Theorem}
		\begin{Theorem}[see Corollary 4.11 in \cite{Kos}]\label{T-Kos}
			Let $p\in(1, 2)$ and let $\mu$ be a probability Borel measure
			on a compact set $\Omega$.
			There is a constant $C:=C(p)$ such that,
			if $X_{n}$ is an $n$-dimensional subspace of~$L_\infty(\mu)$
			such that
			$$
			\|f\|_\infty\le \sqrt{Kn}\|f\|_2\quad \forall f\in X_{n},
			$$
			then
			\begin{multline*}
			\mathbb{E}\sup\limits_{f\in B_p(X_{n})}\Bigl|\frac{1}{m}\sum\limits_{j=1}^m|f(x_j)|^p - \|f\|_p^p\Bigr|
			\\
			\le C\Bigl(\frac{[\log m]^{1+\frac{p}{2}}[\log 4Kn]^{1-\frac{p}{2}}}{m} Kn +
			\Bigl[\frac{[\log m]^{1+\frac{p}{2}}[\log 4Kn]^{1-\frac{p}{2}}}{m} Kn\Bigr]^{1/2}\Bigr),
			\end{multline*}
			where $B_p(X_{n}):=\{f\in X_{n}\colon \|f\|_p\le 1\}$.
		\end{Theorem}
	
	\end{Remark}

\section{Appendix I. Proof of Theorem \ref{p1T2}}\label{sec:7}

Recall the notion of $K$-convexity constant.
Let $(L, \|\cdot\|)$ be a Banach space. The  $K$-convexity constant of the space $L$ is defined as
\begin{align*}
K\bigl(L,\|\cdot\|\bigr):= \sup\Bl( \mathbb{E}_\varepsilon
\Bigl\|\sum\limits_{j=1}^k\varepsilon_j\mathbb{E}_\varepsilon[f(\varepsilon)\varepsilon_j]\Bigr\|^2\Br)^{1/2}
\end{align*}
with  the supremum being  taken over all integers $k\in \mathbb{N}$ and all  functions $f\colon \{-1,1\}^k\to L$ such that $\mathbb{E}_\varepsilon\|f(\varepsilon)\|^2=1$.

\begin{Remark}\label{K-est}
{\rm
It is known (see e.g. Theorem in paragraph 14.6 of \cite{MS})
that there is an absolute constant $C>0$ such that
$K\bigl(L,\|\cdot\|\bigr)\le C\log N$ for any $N$-dimensional
space $(L, \|\cdot\|)$.
Moreover (see \cite[Lemma 17]{JS}),
there is an absolute constant $C>0$ such that
for an $N$-dimensional
subspace
$L\subset L_1(\mu)$ one has $K\bigl(L,\|\cdot\|_1\bigr)\le C\sqrt{\log N}$.
}
\end{Remark}

In \cite{Ta90} the following theorem was proved
(see also \cite[Proposition 15.16]{LedTal} and \cite[Theorem 13]{JS}).

\begin{Theorem}\label{T2}
Let $X_{n}$ be an $n$-dimensional subspace
of $\mathbb{R}^M$
equipped with the counting measure on $\Omega_M:=\{1,\ldots, M\}$.
Assume that there is a number $\theta>0$
such that for each $f\in X_n$ one has
$\|f\|_{\infty}\le \theta\|f\|_{2}$.
Then
$$
\mathbb{E}_\varepsilon\sup\limits_{f\in X_{n}\cap B_1^M}
\Bigl|\sum\limits_{j=1}^M\varepsilon_j|f(j)|\Bigr|\le
2\sqrt{\pi}\theta K\bigl(X_n, \|\cdot\|_{1}\bigr).
$$
\end{Theorem}

The proof of the above theorem relies on the following lemma (see \cite[Proposition 16]{JS}).

\begin{Lemma}\label{K-conv}
Let $X_n$ be an $n$-dimensional subspace of $L_2(\mu)$, where
$\mu$ is a positive but not necessary a probability measure.
Assume that there is a constant $\theta>0$ such that
$\|f\|_\infty\le \theta\|f\|_{{L_2(\mu)}}$ for each $f\in X_n$.
Then
$$
\mathbb{E}_g\sup\limits_{\substack{f\in X_n,\\ \|f\|_{L_1(\mu)}\le1}}
\Bigl|\sum_{k=1}^{n}g_k\langle u_k,f\rangle_{L_2(\mu)}\Bigr|
\le \sqrt{2}\theta K\bigl(X_n,\|\cdot\|_{L_1(\mu)}\bigr),
$$
where $\{u_1, \ldots, u_n\}$ is any orthonormal basis of $X_n$,
and  $g=(g_1, \ldots, g_n)\sim \mathcal{N}(0, I_n)_{\RR^n}$ is the standard Gaussian vector,
i.e. its components are i.i.d. copies of  $Z\sim \mathcal{N}(0,1)$.
\end{Lemma}

Since in \cite{Ta90}, \cite{LedTal}, and \cite{JS}
the statement of Theorem \ref{T2}
is presented in slightly different terms,
we now show how one can prove Lemma \ref{K-conv}
and how to
deduce Theorem \ref{T2} from Lemma \ref{K-conv}.
The proof follows
the argument from the proof of \cite[Proposition 16]{JS}
almost verbatim.

{\bf Proof of Lemma \ref{K-conv}}.
Let $B_1:=\{f\in X_n\colon \|f\|_{L_1(\mu)}\le1\}$.
Let $\{\varepsilon_{ik}\}$, $i\in \{1,\ldots, N\}$, $k\in\{1,\ldots, n\}$
--- be i.i.d. symmetric Bernoulli random variables.
Then by the multivariate central limit theorem, we have
$$
\mathbb{E}_g\sup\limits_{f\in B_1}
\Bigl|\sum_{k=1}^{n}g_k\langle u_k, f\rangle\Bigr|
=
\lim_{N\to \infty}N^{-1/2}\mathbb{E}_\varepsilon\sup\limits_{f\in B_1}
\Bigl|\sum_{k=1}^{n}\sum_{i=1}^{N}\varepsilon_{ik}\langle u_k, f\rangle\Bigr|.
$$
For any fixed vector $\varepsilon=(\varepsilon_{ik})$ there is a vector
$f_\varepsilon\in X_n$
such that
$\|f_\varepsilon\|_1\le1$ and
$$
\sup\limits_{f\in B_1}
\Bigl|\sum_{k=1}^{n}\sum_{i=1}^{N}\varepsilon_{ik}\langle u_k, f\rangle\Bigr|
=
\sum_{k=1}^{n}\sum_{i=1}^{N}\varepsilon_{ik}\langle u_k, f_\varepsilon\rangle.
$$
Thus,  using Fubini's theorem, we get
\begin{align*}
\mathbb{E}_\varepsilon\sup\limits_{f\in B_1}
\Bigl|\sum_{k=1}^{n}\sum_{i=1}^{N}\varepsilon_{ik}\langle u_k, f\rangle\Bigr|
&=\mathbb{E}_\varepsilon
\Bigl[\sum_{k=1}^{n}\sum_{i=1}^{N}\varepsilon_{ik}\langle u_k, f_\varepsilon\rangle\Bigr]\\
&=
\sum_{k=1}^{n}\sum_{i=1}^{N}\langle u_k, f_{ik}\rangle=:S_{N,n},
\end{align*}
where $f_{ik}=\mathbb{E}_\varepsilon[\varepsilon_{ik}f_\varepsilon]$.
By H\"older's inequality, we have
\begin{align*}
S_{N,n}
&=
\int_\Omega \Bigl[\sum_{k=1}^{n}\sum_{i=1}^{N}u_k(x) f_{ik}(x)\Bigr]\, d\mu(x)
\\
&\le
\int_\Omega \Bigl(\sum_{k=1}^{n}\sum_{i=1}^{N}|u_k(x)|^2\Bigr)^{1/2}
\Bigl(\sum_{k=1}^{n}\sum_{i=1}^{N}|f_{ik}(x)|^2\Bigr)^{1/2}\,  d\mu(x)
\\
&\le N^{1/2}
\Bigl\|\Bigl(\sum_{k=1}^{n}|u_k|^2\Bigr)^{1/2}\Bigr\|_\infty
\int_\Omega\Bigl(\sum_{k=1}^{n}\sum_{i=1}^{N}|f_{ik}(x)|^2\Bigr)^{1/2}\,  d\mu(x).
\end{align*}
Since $\{u_1, \cdots, u_n\}$ is an orthonormal basis of $X_n$, we have
\[ \Bigl\|\Bigl(\sum_{k=1}^{n}|u_k|^2\Bigr)^{1/2}\Bigr\|_\infty=\sup_{\sub{ f\in X_n\\
		\|f\|_2=1}}\|f\|_\infty \leq \theta.\]
It then  follows by Khintchine's  inequality  that
\begin{align*}
S_{N,n}&\le
\sqrt{2}N^{1/2}\theta
\int_\Omega\Bigl[ \mathbb{E}_\varepsilon\Bigl|\sum_{k=1}^{N}\sum_{i=1}^{n}\varepsilon_{ik}f_{ik}(x)\Bigr|\Bigr]\,  d\mu(x)
\\
&=
\sqrt{2} N^{1/2}\theta
\mathbb{E}_\varepsilon \Bigl\|\sum_{k=1}^{N}\sum_{i=1}^{n}\varepsilon_{ik}f_{ik}\Bigr\|_1\le
\sqrt{2} N^{1/2}\theta K\bigl(X_n, \|\cdot\|_1\bigr).
\end{align*}
This leads to the desired upper bound:
$$
\mathbb{E}_g\sup\limits_{f\in B_1}
\Bigl|\sum_{k=1}^{n}g_k\langle u_k, f\rangle\Bigr|
=
\lim_{N\to \infty}N^{-1/2} S_{N,n} \leq \sqrt{2} \theta K\bigl(X_n, \|\cdot\|_1\bigr).
$$
The lemma is proved.
\qed

{\bf Proof of Theorem \ref{T2}}.
It is known (see \cite[Theorem 4.12 and Estimate (4.8)]{LedTal})
that
\begin{multline*}
\mathbb{E}_\varepsilon\sup\limits_{f\in X_n\cap B_1^M}
\Bigl|\sum\limits_{j=1}^M\varepsilon_j|f(j)|\Bigr|\le
\sqrt{2\pi}\mathbb{E}_g\sup\limits_{f\in X_n\cap B_1^M}
\Bigl|\sum\limits_{j=1}^M g_jf(j)\Bigr|
=\\=
\sqrt{2\pi}M^{-1}\mathbb{E}_g\sup\limits_{\substack{f\in X_n,\\ \|f\|_{L_1(\mu)}\le1}}
\Bigl|\sum\limits_{j=1}^M g_jf(j)\Bigr|,
\end{multline*}
where $g_1,\ldots, g_m$ are i.i.d. $\mathcal{N}(0,1)$ random variables
and  $\mu$ is
a uniform probability distribution on the set $\Omega_M=\{1, \ldots, M\}$.
For any orthonormal basis $\{u_1, \ldots, u_n\}$ of $X_n$
with respect to the norm $\|f\|_{L_2(\mu)}=\frac{1}{\sqrt{M}}\|f\|_{2}$
we have
\begin{multline*}
\mathbb{E}_g\sup\limits_{\substack{f\in X_n,\\ \|f\|_{L_1(\mu)}\le1}}
\Bigl|\sum\limits_{j=1}^M g _jf(j)\Bigr|
=
\mathbb{E}_g\sup\limits_{\substack{f\in X_n,\\ \|f\|_{L_1(\mu)}\le1}}
\Bigl|\sum\limits_{k=1}^n\langle f, u_k\rangle_{L_2(\mu)}\sum\limits_{j=1}^M g_ju_k(j)\Bigr|
\\
=
\sqrt{M}\, \mathbb{E}_g\sup\limits_{\substack{f\in X_n,\\ \|f\|_{L_1(\mu)}\le1}}
\Bigl|\sum\limits_{k=1}^n G _k\langle f, u_k\rangle_{L_2(\mu)}\Bigr|,
\end{multline*}
where
$G_k := \frac{1}{\sqrt{M}}\sum\limits_{j=1}^M g_ju_k(j)$.
Since $\{u_1, \cdots, u_n\}$ is an orthonormal basis of $(X_n, \|\cdot\|_{L_2(\mu)})$,
it follows by the rotation invariance of the Gaussian random vector that $(G_1, \cdots, G_n)\sim \mathcal{N}(0, I_n)_{\RR^n}$.
In our case we have $$\|f\|_{L_\infty(\mu)}\le \theta\|f\|_2=\theta\sqrt{M}\|f\|_{L_2(\mu)},  \  \   \   \forall f\in X_n.$$
Theorem \ref{T2} then follows  by Lemma \ref{K-conv}.
\qed


\section{Appendix II. Proof of Theorem \ref{thm-5-1}}\label{sec:8}

The proof of Theorem \ref{thm-5-1}  is very close to that of Theorem 16.8.2 of  \cite{Ta}.
We first   recall some notations.
We   identify  a vector from  $\RR^M$ with a function on the set  $\Og_M:=\{1,\cdots, M\}$, and denote by  $\ell_p^M$, $0<p\leq \infty$  the space $\RR^M$ equipped with the norm
$$ \|f\|_p:=\begin{cases}
(\sum_{i=1}^M|f(i)|^p)^{1/p},&\  \ \text{if $0<p<\infty$},\\[5mm]
\max_{1\leq i\leq M} |f(i)|, &\   \ \text{if $p=\infty$}.
\end{cases}$$  Let $B_p^M:=\{f\in \RR^M:\  \|f\|_p\leq 1\}$ denote the unit ball of $\ell_p^M$.
For each $I\subset \Og_M$, we denote by $R_I$ the orthogonal projection onto the space spanned by $e_i, i\in I$, where $e_1=(1,0,\cdots, 0)$, $\cdots$, $e_M=(0, \cdots, 0, 1)\in\RR^M$. Finally,  we set $N_0=1$ and $N_k=2^{2^k}$ for $k=1,2,\cdots$.

The proof of Theorem \ref{thm-5-1} uses  a majorizing measure theorem of Talagrand \cite{Ta}, which we now recall.
Let $(T,\rho)$ be a metric space, and let
$ B_\rho(s, r):=\{t\in T:\  \ \rho(s, t) \leq r\}$ denote the ball  with center $s\in T$ and radius $r>0$.  Let $H_1,\cdots, H_N$ be subsets  of $T$. We say  $H_1,\cdots, H_N$ are $(a, \delta,\rho)$-separated  for some constants  $a>0$ and $0<\delta <1/2$  if  there exist points $t_0\in T$ and  $t_1, \cdots, t_N\in B_\rho(t_0,  4 a)$ such that
\[ \min_{1\leq j\neq k\leq N} \rho(t_j, t_k) \ge a\   \   \ \text{and}\   \   H_j \subset B_\rho (t_j, \delta a)\   \   \ \text{  for all $ 1\leq j \leq N$. }  \]	
A functional $F$  on the metric space  $T$  is a nonnegative function     on  the collection   of all subsets of $T$  satisfying  $ F(H)\leq F(H')$ whenever $H\subset H'\subset T$.
A sequence 	$(F_k)_{k=0}^\infty$ of functionals on $T$  is called decreasing if $F_k(H)\ge F_{k+1}(H)$ for all $k\ge 0$ and $H\subset T$.

The  following version of  the majorizing measure theorem of Talagrand \cite{Ta}  can  be obtained as a combination of \cite[Theorem 2.7.6,  p. 68]{Ta} and   \cite[Theorem 2.2.18,  p. 25]{Ta}.

\begin{Theorem}[Majorizing measure theorem]\textnormal{\cite{Ta}}
	\label{thm-3-3-Ta} Let  $T=(T,\rho)$ be a  metric space with $\diam(T):=\sup_{s, t\in T}\rho(s, t)<\infty$ , and let $(V_t)_{t\in T} $ be a    random process  satisfying  that
	\[ \PP (|V_t-V_s| >u) \leq 2 \exp \Bl( - \f {cu^2} {\rho^2(s, t)}\Br),\   \    \forall u>0, \   \ \forall s, t\in T,\]
	where $c>0$ is an absolute constant.
	Assume that there exists 	a decreasing sequence of functionals $(F_k)_{k=0}^\infty$  on $T$    satisfying   the following growth condition: there exist a  constant   $\sa>0$ and an integer $k_0\ge 2$ such that for  each  integer $k\ge 0$,
	\begin{equation}\label{3-4-Ta}
	F_k\Bl( \bigcup_{\ell=1}^{N_{k+1}} H_\ell\Br) \ge \sa a \sqrt{\log N_{k}} +\min_{1\leq \ell \leq N_{k+1}} F_{k+k_0} (H_\ell)
	\end{equation}
	whenever $H_1, \cdots, H_{N_{k+1}}$ are $(a, 4^{-k_0},\rho)$-separated for some constant $a>0$.
	Then there exists an absolute constant $C>0$ such that \[ \EE \sup_{t\in T} V_t \leq C k_0  \Bl( \f{  F_0(T) }{\sa}+  \diam (T)\Br).\]
\end{Theorem}

To prove Theorem \ref{thm-5-1},  we also  need some known  estimates of  entropy numbers.
Let $(X, \|\cdot\|_X)$ be a Banach space.  Let  $B_X(g,r)$ denote  the   closed ball $\{f\in X:\|f-g\|\le r\}$ with  center $g\in X$ and radius $r>0$.   The entropy numbers $e_k(A,X)$ of  a  set $A$ in $X$ are defined  as
\begin{align*}
e_k(A,\|\cdot\|_X)  :&=\inf \Bl\{\e>0:  \exists g^1,\cdots, g^{2^k}\in A\  \ \text{such that}\  \ A \subset \bigcup_{j=1}^{2^k} B_X(g^j, \va)    \Br \},
\end{align*}
where $k=0,1,\cdots$.
The following   estimates can be found in  \cite[Theorem 2.1 and (2.5)]{DPSTT2}.
\begin{Lemma}\textnormal{\cite[Theorem 2.1 and (2.5)]{DPSTT2}} \label{lem-6-1-0-Ta}
	Let $X_n$ be an $n$-dimensional subspace of $\RR^M$ satisfying
	\[ \|f\|_\infty \leq \sqrt{\f {Kn} M}\|f\|_{2},\   \   \forall f\in X_n\]
	for some constant $K\ge 1$.
	Then for any  $1\leq p\leq 2< q\leq \infty$,
	\begin{equation}\label{8-1a}
	e_k(X_n\cap B_p^M,\   \|\cdot\|_{q})\leq C_{p,q} \Bl( \f {\log M} M\Br) ^{\f 1p-\f1q} \displaystyle \Bl( \f {Kn} k \Br)^{\f 1p-\f1q},\  \ k=1,2,\cdots.
	\end{equation}
\end{Lemma}

Now we turn to the proof of Theorem \ref{thm-5-1}.
Throughout the proof, we use the letter $C_1$ to  denote  a sufficiently large positive constant depending only on $p$, and  use the letter $C$ to denote a general positive constant which  depends only on $p$ but  may vary at each appearance.

Without loss of generality, we may assume that $K=1$ since otherwise we may replace $Kn$ by $n$ and consider $X_n$ as a space of dimension at most $Kn$. We may also assume that  $
M\ge C_1 n \log (2M)
$  since otherwise \eqref{3-1-Ta} holds trivially. In particular, this implies that $\log\log M\leq  \log \f Mn$.

Let  $$T:=\Bl\{ |f|^p:\  \ f\in X_n\cap B_p^M\Br\}\subset B_1^M.$$
For $s\in T$,  define
$V_s:= \sum_{i=1}^M   s(i)\va_i$.
Then $\{V_s\}_{s\in T}$ is a symmetric random process satisfying ({see \cite[Lemma 4.3]{LedTal}}) that
\[ \PP \{ |V_s-V_t|>u\}= \PP \{ |V_{s-t}|>u\}\leq 2 \exp \Bl( -\f { u^2}{2\|s-t\|_2^2}\Br),\   \    \forall s, t\in T,\   \   \forall u>0.\]
Our aim is to show that
\begin{align}
\EE \Bl( \sup_{t\in T  } |V_t |\Br) \leq C \sqrt{\f {n\log M} M} \log \f M{n}.\label{3-1-Ta}
\end{align}

To prove \eqref{3-1-Ta}, we will apply  Theorem \ref{thm-3-3-Ta} to the random process $\{V_s\}_{s\in T}$  on the compact metric space
$(T, \|\cdot\|_2)$.
Since
the condition \eqref{5-1-Ta}  implies  \begin{equation}\label{5-4-Ta}
\|f\|_\infty \leq \Bl( \f {n} M\Br)^{1/p} \|f\|_p,\   \ \forall f\in X_n,
\end{equation}
we have that   for  $t=|f|^p \in T$ with $f\in X_n\cap B_p^M$,
\begin{align*}
\|t\|_2^2  = \sum_{i=1}^M |f(i)|^{2p}\leq \|f\|_\infty^p  \sum_{i=1}^M  |f(i)|^p\leq \f {n} M,
\end{align*}
implying
\begin{equation}\label{3-2-Ta}
\diam (T)=\sup_{s,t\in T}\|s-t\|_2\leq 2\sqrt{\f {n} M}.
\end{equation}

To apply  Theorem \ref{thm-3-3-Ta}, we need to   construct a decreasing sequence of  functionals $(F_k)_{k\ge 0}$  on the metric space $T:=(T, \|\cdot\|_2)$ satisfying  $F_0(T) \leq 2$ and  a growth condition. To this end, for $k=0,1,\cdots$, we let
$
A_k:= \f {2^k} {(\log M)^4}$, and
let $\Sigma_k$ denote the collection of all subsets $I$ of $\Og_M:=\{1,2,\cdots, M\}$ with $|I|\leq A_k$.
For a subset $H\subset T$ and an integer $k\ge 0$, we define $\vi_k(H)=0$ if $A_k<1$, and
\begin{equation}\label{3-5bb}
\vi_k(H): =\max_{I\in\Sigma_k} \inf_{x\in H} \|R_I x\|_{1}
\end{equation}
if $A_k\ge 1$.
Clearly,  $0\leq \vi_k(H)\leq  \vi_{k+1}(H)\leq 1$, and
$\vi_k (H_1)\ge  \vi_k(H_2)\ge \vi_k(T)=0$ whenever $H_1\subset H_2\subset T$.  Now we  define the sequence of  functionals  on $T$ as follows:
\begin{equation}\label{3-5-a} F_k(H):=1-\vi_k(H)+\f 1 {n_1} \max\{n_1-k,0\},\   \   H\subset T,\  \  k=0,1,\cdots,\end{equation}
where  $n_1$ is a positive integer satisfying \begin{equation}\label{3-6a}
C_1\log M\leq n_1<2C_1 \log M.
\end{equation}
Clearly, $(F_k)_{k\ge 0}$ is  a decreasing sequence of  functionals on $T$, and  $F_0(T)\leq 2$.

We need to prove $(F_k)_{k\ge 0}$ satisfies a  growth condition.
Let $k_0\in \NN$ be such that
\begin{equation}\label{5-8-0-Ta}
2^{k_0-1}<\f {C_1 M (\log M)^{6 + \f 4 {2-p}}} { n} \leq 2^{k_0},
\end{equation}
where we  choose the constant $C_1$ large enough so that
\begin{equation}\label{5-8-Ta}
\log_2 \f M{n}\leq k_0 \leq \f{C_1}8 \log_2 \f M{n}.
\end{equation}
Let \begin{equation}\label{5-7-Ta}
\sa:=c_1\sqrt{\f M {n\log M}},\   \    \  \text{where $c_1=\f 1 {4C_1}$. }
\end{equation}
Our aim is to  show  that
for each integer $k\ge 0$,
\begin{equation}\label{5-9-Ta}
F_k\Bl( \bigcup_{\ell=1}^{N_{k+1}} H_\ell\Br) \ge \sa a 2^{k/2} +\min_{1\leq \ell\leq N_{k+1} } F_{k+k_0} (H_\ell)
\end{equation}
whenever $H_1, \cdots, H_{N_{k+1}}\subset T$ are $(a, 4^{-k_0}, \|\cdot\|_{2})$-separated  for some constant  $a>0$.
Once \eqref{5-9-Ta}  is proved, the desired estimate \eqref{3-1-Ta} will follow immediately from  Theorem~\ref{thm-3-3-Ta} and \eqref{5-8-Ta}  since  according to  \eqref{3-2-Ta}, we have  $\diam(T)\leq C \sa^{-1}$.

For the proof of \eqref{5-9-Ta},     we   fix the  integer $k\ge 0$ so that $A_k\ge 1$,  set $N:=N_{k+1}$, and assume that  $H_1, \cdots, H_{N}\subset T$ are $(a, 4^{-k_0}, \|\cdot\|_2)$-separated.
Denote by $B(x, r)$
the Euclidean ball $\{z\in\RR^M:\  \ \|z-x\|_2\leq r\}$ with center  $x\in\RR^M$ and radius $r>0$, and  define   $$B^T(x, r) :=B(x,r)\cap T=\{ z\in T:\  \ \|z-x\|_2\leq r\}\  \ \text{ for $x\in T$ and $r>0$.}$$
Then  there exist  $t_0\in T$ and   $ t_1,\cdots, t_{N} \in  B^T(t_0, 4a)$   such that  $\min_{1\leq i\neq j\leq N} \|t_i-t_j\|_2 \ge a$  and
$H_j\subset B^T(t_j, 4^{-k_0} a)$ for $j=1,\cdots, N$.

We will  reduce   the inequality \eqref{5-9-Ta} to   a relatively simpler  one in several steps.
First,
we   claim that it suffices to  prove \eqref{5-9-Ta} for $ k<  \log_2 (3M)$. Indeed, since  $t_1, \cdots, t_{N}\in B(t_0, 4a)$ are   $a$-separated  with respect to the Euclidean distance (i.e., $\|t_i-t_j\|_2\ge a$ for any $1\leq i\neq j \leq N$),
it follows that   $\log_2 N=2^{k+1}\leq 2M\log_23$, which implies that $k+1=\log_2\log_2 N \leq\log_2( 5M)$.

Second, we   claim  that it suffices to show the inequality
\begin{align} \vi_k\Bl(\bigcup_{j=1}^N H_j\Br) +\sa a 2^{k/2}  \leq
\max_{1\leq j\leq N } \vi_{k+k_0} (H_j) \label{5-10-Ta}
\end{align}
under the conditions
\begin{equation}
0\leq k<  \log_2 (5M)\   \ \text{and}\   \ a^2 2^k \ge \f {n} {M\log M}.\label{5-11-Ta}
\end{equation}
Indeed, by the first claim, we may assume that $0\leq k<  \log_2 (5M)$.  By  \eqref{5-8-Ta} and \eqref{3-6a},  we then have  $k+k_0<n_1$.  Thus, by \eqref{3-5-a},   the desired inequality  \eqref{5-9-Ta}   is equivalent to the inequality,
\begin{align} \vi_k\Bl(\bigcup_{j=1}^N H_j\Br) +\sa a 2^{k/2}  \leq
\max_{1\leq j\leq N } \vi_{k+k_0} (H_j)
+\f {k_0} {n_1}.  \label{5-10-Tb}
\end{align}
This last  inequality   holds trivially if $\sa a 2^{k/2}\leq \f {k_0} {n_1}$ because  $$\vi_k\Bl(\bigcup_{j=1}^N H_j\Br)\leq \min_{1\leq j\leq N} \vi_k(H_j)\leq \min_{1\leq j\leq N} \vi_{k+k_0}(H_j).$$    Thus, we may always assume that
$
\sa a 2^{k/2} >\f {k_0} {n_1}\ge \f 1 {n_1}$,
which,  using \eqref{5-7-Ta} and \eqref{3-6a}, implies  $a^2 2^k \ge \f {n} {M\log M}$.
Thus, it is enough to prove \eqref{5-10-Tb} under the conditions \eqref{5-11-Ta}. Since \eqref{5-10-Ta}  implies \eqref{5-10-Tb}, the claim then follows.

Third,     we claim  that  for the proof of \eqref{5-10-Ta}, it is enough to show  that for  each  set $I\in\Sigma_k$,
\begin{equation}\label{5-13-Ta}
\max_{1\leq \ell\leq N} \max_{\sub{J\in \Sigma_{k+k_0-1}\\
		J\subset \Og_M\setminus I}}\inf_{x\in H_\ell}  \|R_{J} x\|_1 \ge \sa a 2^{k/2}.
\end{equation}
Indeed, by the definition \eqref{3-5bb}, the inequality  \eqref{5-10-Ta} is equivalent to the assertion that for  each  set $I\in\Sigma_k$,
there exist an integer  $1\leq \ell \leq N$ and a set $J\in\Sigma_{k+k_0}$ such that
\begin{align}   \inf_{x\in \bigcup_{j=1}^N H_j } \|R_I x\|_{1}+\sa a 2^{k/2} \leq
\inf_{x\in H_\ell} \|R_{J} x\|_{1}.  \label{5-12-Ta}
\end{align}
Once \eqref{5-13-Ta} is proven, then   for each $I\in \Sigma_k$, there
exist an integer  $1\leq \ell \leq N$ and a set $J_\ell\subset \Og_M\setminus I$ such that $|J_\ell|\leq A_{k+k_0-1}$ and
$\inf_{x\in H_\ell}  \|R_{J_\ell} x\|_1 \ge \sa a 2^{k/2}$.
Setting  $J:=I\cup J_\ell$, we have $$|J|= |I|+|J_\ell|\leq A_{k+k_0-1} +A_k \leq A_{k+k_0},$$
and
\begin{align*}
\inf_{x\in \bigcup_{j=1}^N H_j } \|R_I x\|_{1}+\sa a 2^{k/2}&\leq  \inf_{x\in H_\ell } \|R_I x\|_{1}+\inf_{x\in H_\ell} \|R_{J_\ell} x\|_{1}\\
&\leq
\inf_{x\in H_\ell } \Bl[\|R_I x\|_{1}+\|R_{J_\ell} x\|_{1}\Br]=\inf_{x\in H_\ell} \|R_J x\|_{1},
\end{align*}
proving \eqref{5-12-Ta}.

Finally, we claim  that it is enough to prove  that for each integer  $k$  satisfying  \eqref{5-11-Ta}, and each $I\in \Sigma_k$,
\begin{equation}\label{3-16-Ta}
\max_{1\leq \ell\leq N} \max_{\sub{J\in \Sigma_{k+k_0-1}\\
		J\subset \Og_M\setminus I}}  \|R_{J} t_\ell\|_1 \ge 2 \sa a 2^{k/2} =2c_1\sqrt{\f M {n\log M}} 2^{k/2}a.
\end{equation}
According to  the   second and the third  claims that have already been proven,  we  need to prove that \eqref{3-16-Ta} implies \eqref{5-12-Ta},   assuming $k$ satisfies   \eqref{5-11-Ta}. Indeed, by \eqref{3-16-Ta},
there exist an integer $1\leq \ell \leq N$ and a set $J_\ell\subset\Og_M\setminus  I$  such that $|J_\ell|\leq A_{k+k_0-1}$ and
\begin{align}
\|R_{J_\ell} t_\ell \|_1\ge 2 \sa a 2^{k/2}.\label{3-17-Ta-1}
\end{align}
Since $H_\ell\subset B(t_\ell, 4^{-k_0} a)\cap T$,    it follows that  for any $x\in H_\ell$,
\begin{align*}
&\Bl| 	\|R_{J_\ell} t_\ell \|_1  -	\|R_{J_\ell} x \|_1\Br|\leq 	\|R_{J_\ell} (x-t_\ell)\|_1\leq \sqrt{|J_\ell |}\|t_\ell -x\|_2\\
& \leq \f { 2^{(k+k_0-1)/2} } {(\log M)^2}  4^{-k_0} a\leq a2^{k/2} 2^{-k_0} (\log M)^{-2} \leq \sa    a2^{k/2}.
\end{align*}
Thus, using  \eqref{3-17-Ta-1}, we have
\begin{align*}
&\inf_{x\in H_\ell} \|R_{J_\ell} x \|_1 \ge \|R_{J_\ell} t_\ell\|_1- \sa    a 2^{k/2}\ge \sa  2^{k/2}  a
\end{align*}
implying  \eqref{5-13-Ta}.

In summary, we reduce  to  showing  that  \eqref{3-16-Ta} holds  for each $I\in \Sigma_k$ with   $k$  satisfying  \eqref{5-11-Ta}. This is an easy consequence of the  following proposition.

\begin{Proposition}\label{prop-5-2-Ta} Assume that $1\leq p<2$,  $M\ge C n (\log M)$ for some large constant $C=C_p>1$, and  $k\ge 0$ is an integer satisfying
	$ a^2 2^k \ge \f {n} {M\log M}$ for some constant $a>0$. Assume in addition that  $N=N_{k+1}$ and   $t_1, \cdots, t_{N}$ are points in a ball $B^T(t_0, 4a)$ with $t_0\in T$  satisfying  $\min_{1\leq i\neq  j\leq N}\|t_i-t_j\|_2\ge a$.
	Let
	\begin{equation}\label{5-6-Ta-1} \tau:=\f 1{A_{k+k_0-1}} =2^{-k-k_0 +1}  (\log M)^4\end{equation}
	with  $k_0$ being  the  integer given in  \eqref{5-8-0-Ta} with $C_1=C$.
	Let  $I$ be  a subset of $\Og_M$ satisfying $|I|\leq A_k=2^{k} (\log M)^{-4}$.  For $1\leq \ell \leq N$,
	let $J_\ell:=\Bl\{\Omega_M \setminus I:\  \ t_\ell(i) \ge \tau\Br \}$.
	Then there exists a positive constant $c$ depending only on $p$ such that \begin{equation}\label{5-8-Ta-1}
	\max_{1\leq \ell \leq N} \|R_{J_\ell} t_\ell\|_1 \ge c\cdot  a 2^{k/2} \sqrt{ \f { M} { n \log M}}.
	\end{equation}
\end{Proposition}

The proof of  Proposition \ref{prop-5-2-Ta} is long and  will be given in the next subsection. For the moment, we take it  for granted and proceed with the proof of \eqref{3-16-Ta} for each $I\in \Sigma_k$ with   $k$  satisfying  \eqref{5-11-Ta}.
By \eqref{5-8-Ta-1},  there exists an integer $1\leq \ell \leq N$ such that
\begin{align}
\|R_{J_\ell} t_\ell \|_1\ge c a 2^{k/2}  \sqrt{\f M {n \log M}},\label{5-9-Ta-1}
\end{align}
where  $J_\ell:=\Bl\{\Omega_M\setminus I:\  \ t_\ell(i) \ge \tau\Br \}$, and  $\tau$  is given  in \eqref{5-6-Ta-1}.
Since $\|t_\ell\|_1 \leq 1$, we have
\begin{align*}
&|J_\ell|\leq \f 1\tau =\f { 2^{k+k_0-1}} {(\log M)^4}=A_{k+k_0-1}.
\end{align*}
\eqref{3-16-Ta} then follows with $c_1=\f12 c$.

\subsection{Proof of Proposition \ref{prop-5-2-Ta}}

Recall that   $N=N_{k+1}=2^{2^{k+1}}$,  $1\leq p<2$, and 	$ a^2 2^k \ge \f {n} {M\log M}$.    For each $1\leq \ell \leq N$,  let $f_\ell\in X_n\cap B_p^M$  be such that $t_\ell =|f_\ell|^p$.   Let   $I\subset \Og_M$ be a fixed set  such that    $|I|\leq A_k:= \f {2^k} {(\log M)^4}$, and let  $I^c=\Og_M\setminus I$.  Set \begin{equation}\label{4-1-b}
S: = \max_{1\leq \ell \leq N} \sum_{i \in  I^c } t_{\ell}(i)\chi_{\{ t_\ell(i) \ge \tau\}} (i).
\end{equation}
Our aim is to prove
\begin{equation}\label{4-1:eq}
S\ge   c\cdot  a 2^{k/2} \sqrt{ \f { M} { n \log M}}.
\end{equation}

For the proof of \eqref{4-1:eq}, we need several lemmas.  The first lemma allows us to replace $\{t_\ell\}_{\ell=1}^N$  with   a   subset of points $t_\ell$, $\ell\in V$ that are ``well'' distributed.
\begin{Lemma}\label{lem-6-3-0-Ta}
	We can find a subset $V$ of $\{1, 2,\cdots, N_{k+1}\}$ with $|V|> N_{k-2}$ and the following properties:
	
	\begin{align}
	& \|R_{I^c} (t_\ell-t_{\ell'})\|_2
	\ge \f {a} {\sqrt{2}}\   \  \text{whenever $ \ell,\ell'\in V$ and $  \ell\neq \ell'$},\label{6-4-0-Ta}\\
	& \|f_\ell -f_{\ell'} \|_{\infty} \leq C_p  \Bl( \f {n} {2^k} \cdot \f {\log M}M \Br)^{1/p},\   \   \ \forall \ell, \ell' \in V,\label{6-5-0-Ta}\\
	& \|R_I (f_\ell) -R_{I} (f_{\ell'} ) \|_p  \leq C_p \Bl( \f { n } {M (\log M)^3}\Br)^{1/p}, \   \   \ \forall \ell, \ell' \in V.\label{6-6-0-Ta}
	\end{align}
\end{Lemma}

\begin{proof} First, by Lemma \ref{lem-6-1-0-Ta},
\begin{equation*}
e_{{2^{k-2}}}(X_n\cap B_p^M; \|\cdot\|_{\infty})\leq C_{p} \Bl( \f {\log M} M\Br) ^{\f 1p} \displaystyle \Bl( \f{n}  {2^k} \Br)^{\f 1p}.
\end{equation*}
Since $\{f_1,\cdots, f_N\}\subset X_n\cap B_p^M$,
we can partition $\{1,2,\cdots, N\}$ into at most $N_{k-2}$ sets $V_{\al,1}$, $\al\in \cA_1$ so that  \eqref{6-5-0-Ta} with $V=V_{\al,1}$ is satisfied for each $\al\in\cA_1$.

Second, since  $\|R_I f\|_p \leq |I|^{1/p} \|f\|_\infty\leq A_k^{1/p} \|f\|_\infty$, we have
\begin{align*}
&	e_{{2^{k-2}}} \Bl(R_I X_n \cap B_p^M, \|\cdot\|_p\Br)\leq A_k^{1/p} \cdot e_{2^k} \Bl(R_I X_n \cap B_p^M, \|\cdot\|_\infty\Br)\\
&\leq  A_k^{1/p} \cdot e_{2^k} \Bl( X_n \cap B_p^M, \|\cdot\|_\infty\Br)
\leq C_{p} \Bl( A_k\cdot  \f {\log M } {M}\cdot  \f {n} {2^k} \Br)^{\f 1p}=C_p \Bl( \f {n} {M(\log M)^3}\Br)^{\f1p}.
\end{align*}
Thus, 	we can partition $\{1,2,\cdots, N\}$ into at most $N_{k-2}$ sets $V_{\al,2}$, $\al\in \cA_2$ so that  \eqref{6-6-0-Ta}  with $V=V_{\al,2}$ is satisfied for each $\al\in\cA_2$.

Third, setting  $E:=\{ R_I t_\ell:\  \ \ell=1,2,\cdots, N\}$, we have $$E\subset \{ u\in \RR^{|I|}:\  \ \|u-R_I t_0\|_2\leq 4a\}.$$  Since $|I|\leq A_k$,  $E$ can be covered  by at most $32^{|I|}\leq N_{k-2}$ Euclidean balls of radius $\f a4$ in $\RR^{|I|}$.
Thus, we may partition the set $\{1,2,\cdots, N\}$ into at most $N_{k-2}$ sets $V_{\al,3}$, $\al\in\cA_3$ such that
$$	\max_{\ell, \ell'\in V_{\al,3}} \|R_I (t_\ell) -R_I (t_{\ell'}) \|_2 \leq  \f {a} {4},\  \ \al\in \cA_3.$$
Since $\|t_\ell-t_{\ell'}\|_2\ge a$ for any two distinct $\ell, \ell'\in \{1,\cdots, N\}$, this in turn implies that   \eqref{6-4-0-Ta} with $V={V_{\al,3}}$ is satisfied for each $\al\in\cA_3$.

Finally, let $\{V_\al\}_{\al\in \cA}$ denote the partition of  $\{1,\cdots, N\}$
generated by the above three partitions.  Then $\{V_\al\}_{\al\in \cA}$ contains at most  $N_{k-2}^3 < N_k$ sets  for which    \eqref{6-4-0-Ta},  \eqref{6-5-0-Ta}, and  \eqref{6-6-0-Ta}
with $V=V_\al$ are satisfied for all $\al\in\cA$. Since $N_{k-2}^4=N_k<N=N_{k+1}$, we can find  a set $V$ from this last partition $\{V_\al\}_{\al\in\cA}$  with $|V|>N_{k-2}$.\end{proof}

For the reminder of the proof, we will always use the letter $V$ to denote  a subset of $\{1,\cdots, N\}$ with the properties stated in Lemma \ref{lem-6-3-0-Ta}. Therefore, we will work on the set of points $f_\ell$, $\ell \in V$ instead.

Our second lemma can be used to obtain useful lower estimates of the quantity $S$.

\begin{Lemma}\label{lem-6-4-0-Ta} There exists a constant $c_p>0$ depending only on $p$  such that
	\begin{equation}\label{5-24-Ta}
	\min_{\ell,\ell'\in V,\  \ \ell\neq \ell'}	\Bl[\|f_\ell -f_{\ell'}\|_{\infty} ^p \Br]\cdot S\ge c_p\cdot  a^2.
	\end{equation}
	In particular, this implies
	\begin{equation}\label{6-9-0-0-Ta}
	S \ge c_p a 2^{k/2} \sqrt{\f M {n\log M}} \f 1 {\log M}\ge \f {c_p} {\log ^2M}.
	\end{equation}
\end{Lemma}

\begin{proof}
We will use  the following  inequality.  If $1\leq p\leq 2$, $a, b>0$ and $\eta>0$, then
\begin{equation}\label{4-7-a}
|a^p-b^p|^2 \leq 2^p p^2 ( a_\eta^p+b_\eta^p)|b-a|^p + 2\eta^p (a^p+b^p),
\end{equation}
where $a_\eta =a\cdot  \chi_{[\eta,\infty)} (a)$ and  $b_\eta =b\cdot  \chi_{[\eta,\infty)} (b)$.  To show \eqref{4-7-a}, without loss of generality, we may assume that $a\ge b$.
If $a<\eta$, then
\[	|a^p-b^p|^2 \leq 2 a^p (a^p+b^p) \leq 2\eta^p (a^p+b^p). \]
If $a\ge \eta$, then $a=a_\eta$ and
\begin{align*}
|a^p-b^p|^2 &\leq ( p a^{p-1} |b-a|)^2 =p^2 a^{2p-2} |b-a|^p |b-a|^{2-p} \leq p^2 (2a)^p |b-a|^p\\
& \leq 2^p p^2 ( |a_\eta|^p +|b_\eta|^p) |b-a|^p.
\end{align*}
In either case, we have  \eqref{4-7-a}.

Next, note that  \eqref{4-7-a} implies that for each  $J\subset \Og_M$,  $x, y\in  B_p^M$ and   $\eta\ge  0$,
\begin{align}
\Bl\|  |R_J x|^p -|R_J y|^p \Br\|_{2}^2\leq  4 {\eta^p}  +  C_p\|x-y\|_{\infty}^p A,\label{5-19-0-Ta}
\end{align}
where
\[ A:=\max\Bl\{ \sum_{i\in J} |x(i)|^p\chi_{\{|x(i)|\ge \eta\}}(i),\   \sum_{i\in J}
|y(i)|^p\chi_{\{|y(i)|\ge \eta\}}(i)\Br \}.\]
Thus, setting $\eta=\tau^{1/p}$, and using  \eqref{6-4-0-Ta},  we have that for any $\ell,\ell'\in V$ with $\ell\neq \ell'$,
\begin{align}
\f {a^2} 2 \leq \sum_{i\in \Og_M\setminus I} \Bl||f_\ell (i)|^p-|f_{\ell'} (i)|^p\Br|^2 \leq 4 \tau +C_p\|f_\ell -f_{\ell'}\|_\infty^p S.\label{4-9b}
\end{align}
Since	$ a^2 2^k \ge \f {n} {M\log M}$, we obtain  from \eqref{5-8-0-Ta} with $C=C_1$ that
\[ 4 \tau=2^{-k-k_0+3}(\log M)^4 \leq \f {2^{-k+4}n }{C M (\log M)^{2+\f {4}{2-p}} }\leq \f 14 a^2.\]
 The estimate \eqref{5-24-Ta} then follows from \eqref{4-9b}.

Finally, using \eqref{6-5-0-Ta}, \eqref{5-24-Ta} and the assumption $ a^2 2^k \ge \f {n} {M\log M}$, we obtain
\begin{align*}
S \ge c_p   \f {a^2 2^k m} {n \log M}\ge c_p a 2^{k/2} \sqrt{\f m {n\log M}} \f 1 {\log M}\ge \f {c_p } {(\log M)^2},
\end{align*}
proving \eqref{6-9-0-0-Ta}.
\end{proof}

Our aim is to show \eqref{4-1:eq}, which is an improvement of the first inequality in \eqref{6-9-0-0-Ta}.
According to \eqref{5-24-Ta}, for the proof of  \eqref{4-1:eq},   it will suffice to show that  there exist  two distinct $\ell, \ell'\in V$ such that
\begin{equation}\label{4-10a}
\|f_\ell -f_{\ell'}\|_{\infty} ^p \leq C_p  S  \f {n} {2^k} \f {\log M} M.
\end{equation}
The idea is to construct a set $U=U_S$ with the following two properties:
\begin{enumerate}[\rm (i)]
\item  $\{f_\ell-f_{\ell_0}:\  \ \ell\in V\} \subset U$ for some $\ell_0\in V$;

	\item  there exists a partition $\{E_\al\}_{\al\in \cA}$ of $U$ such that $|\cA|<|V|$ and  $$\diam(E_\al, \|\cdot\|_\infty)\leq  C_p  \Bl(S  \f {n} {2^k} \f {\log M} M\Br)^{1/p},\   \  \forall \al\in\cA.$$
\end{enumerate}
Since $|V|>N_{k-2}$, to ensure  (ii), it suffices to prove that
	$$ e_{2^{k-2}}(U, \|\cdot\|_\infty)\leq  C_p  \Bl(S  \f {n} {2^k} \f {\log M} M\Br)^{1/p}.$$

To show \eqref{4-10a}, we need two additional lemmas. For $\xi\ge 0$, we define
\begin{align*}
U(\xi) =\xi B^M_p +\tau^{\f 1p-\f 12} B_2^M\  \ \text{and}\   \   X_n(\xi) =U(\xi)\cap X_n.
\end{align*}

\begin{Lemma}\label{lem-6-5-0-Ta} Let $\xi=(3S)^{1/p}$. Then
	\begin{equation}
	\{ f_\ell-f_{\ell'}:\  \ \ell, \ell'\in V \}\subset  2 X_n(\xi).
	\end{equation}
\end{Lemma}

\begin{proof}

Fix $\ell, \ell'\in V$ and set $g:=g_{\ell, \ell'} =f_\ell-f_{\ell'}\in X_n$.  Our aim is  to show that $g\in 2 U(\xi)$.
First,  by \eqref{6-6-0-Ta} and \eqref{6-9-0-0-Ta}, we have
\begin{align*}
\sum_{i\in I} |g (i)|^p=\sum_{i\in I} |f_\ell (i) -f_{\ell'} (i)|^p \leq \f {C_p n } {M(\log M)^3}\leq  S,
\end{align*}
implying $g\chi_I \in S^{1/p}\cdot  B_p^M$.
Second,  setting $\eta=\tau^{1/p}$, and using \eqref{4-1-b}, we have
\begin{align*}
2S &\ge  \sum_{i \in  \Og_M\setminus  I } |f_{\ell}(i)|^p  \chi_{\{ |f_\ell(i)| \ge \eta\}} (i)+  \sum_{i \in  \Og_M\setminus  I } |f_{\ell'}(i)|^p  \chi_{\{ |f_{\ell'}(i)| \ge \eta\}} (i)\\
&\ge 2^{-p}  \sum_{i \in  \Og_M\setminus  I } |g(i)|^p  \chi_{\{ |g(i)| \ge 2\eta\}} (i),
\end{align*}
where the last step
uses the inequality,
\[ |u-v|^p \chi_{|u-v|\ge 2\eta} \leq 2^p \Bl( |u|^p \chi_{|u|\ge \eta}  +|v|^p \chi_{|v|\ge \eta}\Br),\  \ u, v\in\RR. \]
This implies that
\[ g \chi_{_{\{i\in I^c:\ \ |g(i)|\ge 2\eta\}}}\in 2 \cdot (2S)^{1/p} B_p^M. \]

Finally, we write $g=u+v$, where
\begin{align*}
v:&= g\chi_I+ g \chi_{_{\{i\in I^c:\ \ |g(i)|\ge 2\eta\}}},\  \ \text{and}\   \ u:=  g  \chi_{\{i\in I^c:\   |g(i)| < 2\eta\}}.
\end{align*}
Clearly,  $\f 12 v\in (3S)^{1/p}\cdot B_p^M$ and
$$\f u2 \in (\eta B_\infty^M) \cap B_p^M\subset \eta^{1-\f p2} \cdot B_2^M.$$
Thus, $g\in 2 U(\xi)$.
\end{proof}

%
%
%
%
%
%
%

\begin{Lemma}\label{prop-6-7-0-Ta} Let  $\xi>0$ be such that    \begin{equation}\label{6-11-0-Ta}
	\tau \leq \xi^{\f {2p} {2-p}} \Bl(\f {n} {2^k} \f {\log M} M\Br).
	\end{equation}
	Then
	\begin{equation}\label{6-12-0-Ta} e_{j}\Bl(X_n(\xi), \|\cdot\|_{\infty}\Br)\leq C \xi \cdot \Bl(\f {n} {j} \f {\log M} M\Br)^{1/p},\   \   j=1,2,\cdots, 2^k. \end{equation}
\end{Lemma}

\begin{proof}
Note first that  $U(\xi)=\xi B_p^M+ \tau^{\f 1p-\f12} B_2^M$ is a symmetric convex body in $\RR^M$.
Let $W$ denote the polar of $U(\xi)$; that is,
\[ W:=\Bl\{x\in\RR^M:\   \   \|x\|_W:= \max_{y\in U(\xi)} x\cdot y \leq 1\Br\}.\]
Then  for each $x\in W$, \begin{align*}
\|x\|_W&=\sup_{\|u\|_p\leq \xi} \sup_{\|v\|_2\leq \tau^{\f1p-\f12}} (x\cdot u +x\cdot v)  \leq \xi\|x\|_{p'} +\tau^{\f1p-\f 12} \|x\|_2.
\end{align*}
By Lemma \ref{lem-6-1-0-Ta}, this implies  that  for $1\leq j\leq 2^k$,
\begin{align*}
&e_{j} (B_2^M \cap X_n, \|\cdot\|_W) \leq \xi\cdot  e_{j} (B_2^M\cap X_n , \|\cdot\|_{p'}) +2\tau^{\f1p-\f 12}\\
&\leq C \xi\cdot \Bl(\f {n}{j} \f { \log M} M\Br) ^{\f 1p-\f12} +2 \tau^{\f1p-\f 12}\leq C \xi\cdot \Bl(\f {n}{j} \f { \log M} M\Br) ^{\f 1p-\f12},
\end{align*}
where the last step uses \eqref{6-11-0-Ta}.
By duality (see \cite[Theorem 16.8.10]{Ta}), we deduce
\begin{equation}\label{6-13-0-Ta}
e_j( X_n(\xi), \|\cdot\|_2) \leq C\xi\cdot \Bl(\f {n}{j} \f { \log M} M\Br) ^{\f 1p-\f12},\   \   j=1,2,\cdots, 2^k.
\end{equation}
Finally, using \eqref{6-13-0-Ta} and Lemma \ref{lem-6-1-0-Ta},
we  have  that for $1\leq j\leq 2^{k-1}$,
\[ e_{2j} (X_n(\xi), \|\cdot\|_\infty) \leq 2 e_j (X_n(\xi), \|\cdot\|_2) e_j(X_n\cap B_2^M, \|\cdot\|_\infty)\leq C\xi\cdot \Bl( \f {n}j \f {\log M}M\Br)^{\f 1p}.  \]
The stated estimate then follows by monotonicity.
\end{proof}

Now we are in a position to prove Proposition \ref{prop-5-2-Ta}.

\begin{proof}[{Proof of Proposition \ref{prop-5-2-Ta}}] Let $\xi=(3S)^{1/p}$.
A straightforward calculation using \eqref{6-9-0-0-Ta} and \eqref{5-8-0-Ta} shows that  the condition \eqref{6-11-0-Ta}  is satisfied.  Consequently, using Lemma \ref{prop-6-7-0-Ta}, we obtain
\begin{equation}\label{6-16-0-Ta} e_{2^{k-2}}\Bl(X_n(\xi), \|\cdot\|_{\infty}\Br)\leq C \xi \cdot \Bl(\f {n} {2^k} \f {\log M} M\Br)^{1/p}.  \end{equation}
Thus, we may partition the set $2 X_n(\xi)$ into $N_{k-2}$ sets $E_\ga$, $\ga\in\Ga$ such that
\begin{equation}\label{4-17a}
\diam(E_\ga, \|\cdot\|_\infty) \leq C	\xi \cdot \Bl(\f {n} {2^k} \f {\log M} M\Br)^{1/p},\  \  \ \forall \ga\in\Ga.
\end{equation}

On the other hand, Lemma \ref{lem-6-5-0-Ta} implies that  for each fixed $\ell_0\in V$,
\[ \{ f_\ell-f_{\ell_0}:\  \ \ell \in V\}\subset 2 X_n(\xi).\]
Since $|V|> N_{k-2}$ and $|\Ga|= N_{k-2}$, we can find two distinct $\ell,\ell'\in V$ such that  $f_\ell-f_{\ell_0}, f_{\ell'}-f_{\ell_0}$ lie in a same set of the partition $\{E_\ga\}_{\ga\in\Ga}$, which, using \eqref{4-17a}, implies
\[ \|f_\ell -f_{\ell'}\|_\infty \leq C	\xi \cdot \Bl(\f {n} {2^k} \f {\log M} M\Br)^{1/p}.\]
Thus, there exist two distinct   $\ell, \ell'\in V$ such that    the estimate \eqref{4-10a} holds, which, using \eqref{5-24-Ta}, in turn implies  the desired estimate \eqref{4-1:eq}.
\end{proof}

\section*{Acknowledgement}
The authors would like to thank the referee very much for  careful reading of their paper and many  helpful  suggestions and comments.

\Addresses

\end{document}